\documentclass[12pt]{amsart}
\usepackage{amsrefs}

\usepackage[utf8]{inputenc}
\usepackage[T1]{fontenc}

\usepackage{amsmath, amsthm, amssymb, mathrsfs,amsrefs}
\usepackage{enumerate}
\usepackage{tikz-cd}
\usetikzlibrary{arrows}

\usepackage{hyperref}
\definecolor{darkblue}{rgb}{0.0,0.0,0.3}
\hypersetup{colorlinks,breaklinks,linkcolor=red,urlcolor=red,anchorcolor=red,citecolor=red}

\usetikzlibrary{matrix,arrows,decorations.pathmorphing}

\let\temp\phi
\let\phi\varphi
\let\varphi\temp

\DeclareMathOperator{\Aut}{Aut}

\newcommand{\C}{\mathbb{C}}

\newcommand{\N}{\mathbb{N}}
\newcommand{\R}{\mathbb{R}}

\newcommand{\calB}{\mathcal{B}}

\newcommand{\calK}{\mathcal{K}}

\newcommand{\calM}{\mathcal{M}}

\newcommand{\rmA}{\mathrm{A}}
\newcommand{\rmB}{\mathrm{B}}
\newcommand{\rmC}{\mathrm{C}}

\DeclareMathOperator{\id}{id}

\DeclareMathOperator{\ca}{\mathrm{C}^*}
\DeclareMathOperator{\camin}{\mathrm{C}_{\mathrm{min}}^*}

\DeclareMathOperator{\ucp}{UCP}
\DeclareMathOperator{\ccp}{CCP}

\DeclareMathOperator{\NCConv}{NCConv}
\DeclareMathOperator{\PoNCConv}{PoNCConv}
\DeclareMathOperator{\UnOpSys}{UnOpSys}
\DeclareMathOperator{\OpSys}{OpSys}

\newcommand{\bC}{\mathbb{C}}

\usepackage{xspace}
\newcommand{\csharp}{$\text{C}^\sharp$\xspace}

\theoremstyle{plain}
\newtheorem{theorem}{Theorem}[section]
\newtheorem{lemma}[theorem]{Lemma}

\newtheorem{proposition}[theorem]{Proposition}
\newtheorem{corollary}[theorem]{Corollary}

\theoremstyle{definition}

\newtheorem{example}[theorem]{Example}
\newtheorem{remark}[theorem]{Remark}
\newtheorem{definition}[theorem]{Definition}

\newtheorem{introtheorem}{Theorem}

\newcommand{\Irr}{\operatorname{Irr}}
\newcommand{\Spec}{\operatorname{Spec}}

\DeclareMathOperator{\mintens}{\otimes_{\mathrm{min}}}

\renewcommand{\sb}{\overline{\partial K}}

\numberwithin{equation}{section}

\begin{document}
%%%

%%%
\title[Nonunital operator systems and nc convexity]{Nonunital operator systems and noncommutative convexity}

\author[M. Kennedy]{Matthew Kennedy}
\address{Department of Pure Mathematics\\University of Waterloo\\Waterloo, ON, N2L 3G1, Canada}
\email{matt.kennedy@uwaterloo.ca}
\author[S.J. Kim]{Se-Jin Kim}
\address{School of Mathematics and Statistics\\University of Glasgow\\University Place\\Glasgow, G12 8QQ\\United Kingdom}
\email{sejin.kim@glasgow.ac.uk}
\author[N. Manor]{Nicholas Manor}
\address{Department of Pure Mathematics\\ University of Waterloo\\Waterloo, ON, N2L 3G1, Canada}
\email{nmanor@uwaterloo.ca}

\begin{abstract}
We establish the dual equivalence of the category of (potentially non-unital) operator systems and the category of pointed compact nc (noncommutative) convex sets, extending a result of Davidson and the first author. We then apply this dual equivalence to establish a number of results about operator systems, some of which are new even in the unital setting.

For example, we show that the maximal and minimal C*-covers of an operator system can be realized in terms of the C*-algebra of continuous nc functions on its nc quasistate space, clarifying recent results of Connes and van Suijlekom. We also characterize ``C*-simple'' operator systems, i.e.\ operator systems with simple minimal C*-cover, in terms of their nc quasistate spaces.

We develop a theory of quotients of operator systems that extends the theory of quotients of unital operator algebras. In addition, we extend results of the first author and Shamovich relating to nc Choquet simplices. We show that an operator system is a C*-algebra if and only if its nc quasistate space is an nc Bauer simplex with zero as an extreme point, and we show that a second countable locally compact group has Kazhdan's property (T) if and only if for every action of the group on a C*-algebra, the set of invariant quasistates is the quasistate space of a C*-algebra.
\end{abstract}

\subjclass[2010]{}
\thanks{First author supported by Canadian Natural Sciences and Engineering Research Council (NSERC) Discovery Grant number 2018-202107.}
\thanks{Second author supported by Canadian Natural Sciences and Engineering Research Council (NSERC) PGS-D Scholarship number 396162013, and the European Research Council (ERC) Grant number 817597 under the European Union's Horizon 2020 research and innovation programme.}
\thanks{Third author supported by the Canadian Natural Sciences and Engineering Research Council (NSERC) PGS-D Scholarship number 401226864.}

\maketitle

\setcounter{tocdepth}{1}
\tableofcontents
%%%

%%%
\section{Introduction}
%%%

Werner's notion of a (generalized, i.e.\ potentially nonunital) operator system is an axiomatic, representation-independent characterization of concrete operator systems, which are self-adjoint subspaces of bounded operators acting on a Hilbert space. Werner \cite{Wer2002} showed that every concrete operator system satisfies the axioms of an abstract operator system, and conversely that every abstract operator system is isomorphic to a concrete operator system, thereby generalizing an important result of Choi and Effros \cite{CE1977} for unital operator systems. 

Recently, Davidson and the first author \cite{DK2019} introduced a theory of noncommutative convex sets and noncommutative functions. A key starting point for the theory is the dual equivalence between the category of compact noncommutative convex sets and the category of closed unital operator systems. On the one hand, this equivalence allows the rich theory of operator systems and C*-algebras to be applied to problems in noncommutative convexity. On the other hand, recent results suggest that that the perspective of noncommutative convexity can also provide new insight on operator systems and C*-algebras (see e.g.\  \cites{DK2015, KS2019, DK2020}).

In this paper we will establish a similar dual equivalence between the category of operator systems in the sense of Werner and a category of objects that we call pointed noncommutative convex sets. These are certain pairs consisting of a compact noncommutative convex set along with a distinguished point in the set. We will then consider a number of applications of this equivalence.

Before stating our results, we will first briefly review some of the basic ideas from the theory of noncommutative convexity.

A compact nc (noncommutative) convex set is a graded set $K = \sqcup K_n$, where each graded component $K_n$ is an ordinary compact convex subset of the set $\calM_n(E)$ of $n \times n$ matrices over an operator space $E$, and the graded components are related by requiring that $K$ is closed under direct sums and compressions. The union is taken over all $n \leq \kappa$ for some sufficiently large infinite cardinal number $\kappa$ depending on $K$. The fact that $\kappa$ is infinite is an essential part of the theory, being necessary for e.g.\ the existence of extreme points. In the separable setting, it typically suffices to take $\kappa = \aleph_0$. 

The conditions on $K$ are equivalent to requiring that $K$ is closed under nc convex combinations, meaning that $\sum \alpha_i^* x_i \alpha_i \in K_n$ for every bounded family of points $\{x_i \in K_{n_i}\}$ and every family of scalar matrices $\{\alpha_i \in \calM_{n_i,n}\}$.

The prototypical example of a compact nc convex set is the nc state space of a unital operator system $S$ defined by $K = \sqcup K_n$, where $K_n = \ucp(S,\calM_n)$ is the set of unital completely positive maps from $S$ into the space $\calM_n$ of $n \times n$ matrices. The dual equivalence in \cite{DK2019} implies that $S$ is isomorphic to the operator system $\rmA(K)$ of continuous affine nc functions on $K$, and that, on the other hand, if $K$ is a compact nc convex set, then $K$ is affinely homeomorphic to the nc state space of the operator system $\rmA(K)$. In particular, every compact nc convex set arises as the nc state space of an operator system.

For a (generalized) operator system $S$, it is necessary to instead consider the nc quasistate space of $S$. This is the pair $(K,z)$ consisting of the compact nc convex set $K = \sqcup K_n$, where $K_n = \ccp(S,\calM_n)$ is the set of completely contractive and completely positive maps from $S$ into $\calM_n$, and $z \in K_1$ is the zero map.

We are therefore led to consider pairs $(K,z)$ consisting of a compact nc convex set $K$ and a distinguished point $z \in K_1$. However, it turns out that not every pair $(K,z)$ arises as the nc quasistate space of an operator system. This is an important point that explains many of the difficulties that arise in the non-unital setting. In order to obtain the desired dual equivalence between operator systems and pointed compact nc convex sets, it is necessary to impose an additional constraint.

Specifically, we say that the pair $(K,z)$ is a pointed compact nc convex set if the operator system $\rmA(K,z) \subseteq \rmA(K)$ of continuous affine nc functions on $K$ that vanish at $z$ has nc quasistate space $(K,z)$. Our results will imply that this property is equivalent to $(K,z)$ arising as the state space of a compact nc convex set.

We consider pointed compact nc convex sets and functions on pointed compact nc convex sets in Section \ref{sec:pointed-nc-convex-sets} and Section \ref{sec:pointed-nc-functions} respectively. The following two results establishing the above-mentioned dual equivalence are the main results in Section \ref{sec:categorical-duality}.

\begin{introtheorem} \label{thm:intro-duality}
	An operator system $S$ with nc quasistate space $(K,z)$ is isomorphic to the operator system $\rmA(K,z) \subseteq \rmA(K)$ of continuous affine nc functions on $K$ that vanish at $z$. Hence $(K,z)$ is a pointed compact nc convex set if and only if it arises as the nc quasistate space of an operator system.
\end{introtheorem}

Theorem \ref{thm:intro-duality} is the key ingredient in the dual equivalence between the category of generalized operator systems and the category of pointed compact nc convex sets.

\begin{introtheorem}
The category $\OpSys$ of generalized operator systems is dually equivalent to the category $\PoNCConv$ of pointed compact nc convex sets. 
\end{introtheorem}

An important consequence of Theorem \ref{thm:intro-duality} is that essentially all of the results from \cite{DK2019} about unital operator systems apply to (generalized) operator systems. For example, in Section \ref{sec:min-max-c-star-covers}, we establish characterizations of the maximal and minimal C*-covers of an operator system in terms of the C*-algebra of continuous nc functions on its nc quasistate space. As a corollary, we recover results about the minimal C*-cover (i.e.\ the C*-envelope) recently obtained by Connes and van Suijlekom \cite{CS2020}. 

\begin{introtheorem}
Let $(K,z)$ be a pointed compact nc convex set.
\begin{enumerate}
\item The C*-algebra $\rmC(K,z)$ of pointed continuous nc functions on $(K,z)$ is the maximal C*-cover of $\rmA(K,z)$.
\item Let $I_{\sb}$ denote the boundary ideal in the C*-algebra $\rmC(K)$ of continuous nc functions on $K$ relative to the unital operator system $\rmA(K)$, so that the C*-algebra $\rmC(K)/I_{\sb} \cong \rmC(\sb)$ is the minimal unital C*-cover of $\rmA(K)$, and let $I_{(\sb,z)} = I_{\sb} \cap \rmC(K,z)$. Then the C*-algebra $\rmC(K,z)/I_{(\sb,z)}$ is the minimal C*-cover of $\rmA(K,z)$.
\end{enumerate}
\end{introtheorem}

In Section \ref{sec:quotients-operator-systems}, as another application of the dual equivalence between operator systems and pointed compact nc convex sets, we develop a theory of quotients of generalized operator systems that extends the theory of quotients of unital operator systems developed by Kavruk, Paulsen, Todorov and Tomforde \cite{KPTT2013}.

\begin{introtheorem} \label{thm:functorial-characterization-quotients}
Let $S$ be an operator system and let $J \subseteq S$ be the kernel of a completely contractive and completely positive map on $S$. There is a unique pair $(S/J,\phi)$ consisting of an operator system $S/J$ and a morphism $\phi : S \to S/J$ with the property that whenever $T$ is an operator system and $\psi : S \to T$ is a completely contractive and completely positive map with $J \subseteq \ker \psi$, then $\psi$ factors through $\phi$. In other words, there is a completely contractive and completely positive map $\omega : S/J \to T$ such that $\psi = \omega \circ \phi$.
\end{introtheorem}

We also obtain some results that are new even for unital operator systems. In Section \ref{sec:c-star-simplicity}, we establish a characterization of operator systems that are C*-simple, meaning that their minimal C*-cover is simple. We refer to Section \ref{sec:min-max-c-star-covers} for the definition of the spectral topology.

\begin{introtheorem}
An operator system $S$ with nc quasistate space $(K,z)$ is C*-simple if and only if the closed nc convex hull of any nonzero point in the spectral closure of $\partial K$ contains $\partial K \setminus \{z\}$.
\end{introtheorem}

In Section \ref{sec:characterization-c-star-algebras}, we establish a characterization of operator systems that are isomorphic to C*-algebras in terms of their nc quasistate spaces, extending a result for unital operator systems from \cite{KS2019}.

\begin{introtheorem}
Let $S$ be an operator system with nc quasistate space $(K,z)$. Then $S$ is a C*-algebra if and only if $K$ is an nc Bauer simplex and $z$ is an extreme point. The result also holds for unital operator systems with nc quasistate spaces replaced by nc state spaces.	
\end{introtheorem}

In Section \ref{sec:stable-equivalence}, we make another connection to the recent work of Connes and van Suijlekom \cite{CS2020}. They consider operator systems $S$ and $T$ that are stably equivalent in the sense that the minimal tensor product $S \mintens \calK$ is isomorphic to the minimal tensor product $T \mintens \calK$, where $\calK$ denotes the C*-algebra of compact operators. The next result is a characterization of stable equivalence of operator systems in terms of their nc quasistate spaces.

\begin{introtheorem}
	Let $S$ and $T$ be operator systems with nc quasistate spaces $(K,z)$ and $(L,w)$ respectively. Let $0_\calK$ and $\id_{\calK}$ denote the zero map and the identity representation respectively of $\calK$. Then $S$ and $T$ are stably isomorphic if the closed nc convex hulls of $\partial K \otimes \{0_{\calK}, \id_{\calK}\}$ and $\partial L \otimes \{0_{\calK}, \id_{\calK}\}$ are pointedly affinely homeomorphic with respect to the points $z \otimes 0_\calK$ and $w \otimes 0_\calK$ (see Section \ref{sec:stable-equivalence}). 
\end{introtheorem}

Finally, in Section \ref{sec:dynamics} we establish the following characterization of second countable locally compact groups with property (T), extending a result from \cite{KS2019} for discrete groups acting on unital C*-algebras, as well as a result of Glasner and Weiss from \cite{GW1997} for second countable locally compact groups acting on unital commutative C*-algebras.

\begin{introtheorem}
	A second countable locally compact group $G$ has Kazhdan's property (T) if and only if for every action of the group on a C*-algebra, the set of invariant quasistates is the quasistate space of a C*-algebra. The result also holds for unital C*-algebras with quasistate spaces replaced by state spaces. 
\end{introtheorem}

%%%
\section{Preliminaries on operator systems} \label{sec:generalized-operator-systems}
%%%

In this section we will recall the notion of a matrix ordered operator space and introduce the notion of a generalized (i.e.\ potentially nonunital) operator system. For a reference on operator spaces and unital operator systems, we refer the reader to the books of Paulsen \cite{Pau2002} and Pisier \cite{Pis2003}.

Let $E$ be a self-adjoint operator space, i.e.\ such that $E = E^*$. We let $E_h = \{x \in E : x = x^*\}$ denote the set of self-adjoint elements in $E$. For $n \in \N$, we will write $\calM_n(E)$ for the operator space of $n \times n$ matrices over $E$, and we will write $\calM_n$ for $\calM_n(\C)$. A {\em matrix cone} over $E$ is a disjoint union $P = (P_n)_{n \in \mathbb{N}}$ of closed subsets $P_n \subseteq M_n(E)_h$ such that
\begin{enumerate}
\item $P_n \cap -P_n = 0$ for all $n \in \mathbb{N}$ and
\item $AP_nA^* \subseteq P_m$ for all $A \in M_{m,n}$ and $m,n \in \mathbb{N}$.
\end{enumerate}

\begin{definition}
A {\em matrix ordered operator space} is a pair $(E,P)$ consisting of a self-adjoint *-vector space $E$ and a matrix cone $P$ over $E$. For $n \in \N$, an element in $\calM_n(E)$ is {\em positive} if it belongs to $P_n$.
\end{definition}

\begin{remark}
When referring to a matrix ordered operator space, we will typically omit the positive cone unless we need to refer to it explicitly. Note that if $E$ is a matrix ordered operator space then for $m \in \mathbb{N}$, the space $\calM_m(E)$ is a matrix ordered operator space in a canonical way. Specifically, letting $P$ denote the matrix cone for $E$, $(\calM_m(E),Q)$ is a matrix ordered operator space, where $Q = (Q_n)_{n \in \mathbb{N}}$ is the matrix cone defined by identifying $\calM_m(\calM_n(E))$ with $\calM_{mn}(E)$ in the obvious way and setting $Q_n = P_{mn}$.
\end{remark}

Let $E$ be a matrix ordered operator space. An element $e \in E$ is an {\em archimedean order unit} for $E$ if for every $x \in E_h$, there is a scalar $\alpha > 0$ such that $-\alpha e \leq x \leq \alpha e$, and if $x + \alpha e \geq 0$ for all $\alpha > 0$, then $x \geq 0$. It is an {\em archimedean matrix order unit} for $E$ if for every $n \in \N$, $1_n \otimes e$ is an archimedean order unit for $\calM_n(E)$.

If $E$ is a matrix ordered operator space, then an archimedean matrix order unit $e \in E$ induces a norm $\|\cdot\|_e$ on $\calM_n(E)$ for each $n \in \N$, defined by
\[
\|x\|_e = \inf \left\{ \alpha > 0 : \left(\begin{matrix} \alpha 1_n \otimes e & x \\ x^* & \alpha 1_n \otimes e \end{matrix}\right) \geq 0 \right\} \quad \text{for} \quad x \in \calM_n(E).
\]

The next definition is equivalent to the definition of an operator system given by Choi and Effros \cite{CE1977}.

\begin{definition} \label{defn:unital-operator-system}
A {\em unital operator system} $S$ is a complete matrix ordered operator space with an archimedean matrix order unit $1_S$ that is {\em distinguished} in the sense that for each $n$, the norm on $\calM_n(S)$ coincides with the norm $\|\cdot\|_{1_S}$ from above.
\end{definition}

\begin{remark}
Although not strictly necessary, it will be convenient for the purposes of this paper to assume that operator systems are complete. If $S$ is a unital operator system, then the distinguished archimedean order unit $1_S$ is uniquely determined by the property that for $s \in S$ with $s \geq 0$, $\|s\| \leq 1$ if and only if $s \leq 1_S$.
\end{remark}

Let $(E,P)$ and $(F,Q)$ be matrix ordered operator spaces and let $\phi : E \to F$ be a bounded map. We will write $\phi_n : \calM_n(E) \to \calM_n(F)$ for the linear map defined by $\phi_n = \id_n \otimes \phi$. 

\begin{definition}
Let $(E,P)$ and $(F,Q)$ be matrix ordered operator spaces. A linear map $\phi : E \to F$ is {\em contractive} if $\|\phi\| \leq 1$, and {\em completely contractive} if $\phi_n$ is contractive for all $n \in \mathbb{N}$. It is {\em isometric} if $\|\phi(x)\| = \|x\|$ for all $x \in E$, and {\em completely isometric} if $\|\phi_n(x)\| = \|x\|$ for all $n \in \N$ and all $x \in \calM_n(E)$. Similarly, it is {\em positive} if $\phi(P_1) \subseteq Q_1$, and {\em completely positive} if $\phi_n$ is positive for each $n \in \mathbb{N}$. The map $\phi$ is a {\em complete order isomorphism} if it is completely positive and invertible with a completely positive inverse. It is a {\em complete order embedding} if it is completely positive and invertible on its range with a completely positive inverse.
\end{definition}

\begin{remark}
For unital operator systems, these definitions agree with the usual definitions. Furthermore, because the norm on a unital operator system is completely determined by the matrix order, a unital map between unital operator systems is completely isometric if and only if it is a complete order embedding. However, this is not true for arbitrary matrix ordered operator spaces (see \cite{Wer2002}).
\end{remark}

We will write $\UnOpSys$ for the category of unital operator systems with unital completely positive maps (equivalently, unital complete order homomorphisms) as morphisms. We will refer to unital complete order isomorphisms as {\em isomorphisms}, and to unital complete order embeddings as {\em embeddings}. 

Choi and Effros \cite{CE1977}*{Theorem 4.4} showed that every unital operator system is isomorphic to a concrete unital operator system, meaning that there is a unital completely isometric map into some $\calB(H)$, where $\calB(H)$ denotes the C*-algebra of bounded linear operators acting on a Hilbert space $H$. We will be interested in matrix ordered operator spaces satisfying an appropriate analogue of this property. 

Specifically, we are interested in matrix ordered operator spaces with a completely isometric complete order embedding into some $\calB(H)$. It turns out that not every matrix ordered operator space has this property. Following Connes and van Suijlekom \cite{CS2020}, we will make use of Werner's \cite{Wer2002} characterization of matrix ordered operator spaces with this property in terms of partial unitizations (see below), although other characterizations are also known (see e.g.\ \cite{Rus2015}).

The next definition is \cite{Wer2002}*{Definition 4.1} (see also \cite{CS2020}*{Definition 2.11}).

\begin{definition} \label{defn:partial-unitization}
Let $E$ be a matrix ordered operator space. The {\em partial unitization} of $E$ is the matrix ordered operator space $(E^\sharp,P)$, where $E^\sharp = E \oplus \bC$ and the matrix cone $P = (P_n)$ is defined by specifying that for each $n \in \N$, $P_n \subseteq \calM_n(E^\sharp)_h = \calM_n(E)_h \oplus (\calM_n)_h$ consists of all pairs $(x,\alpha) \in \calM_n(E)_h \oplus (\calM_n)_h$ satisfying
\[
\alpha \geq 0 \text{ and } \phi(\alpha_\epsilon^{-1/2} x \alpha_\epsilon^{-1/2}) \geq -1 \text{ for all } \epsilon > 0 \text{ and } \phi \in \ccp(E,\calM_n),
\]
where $\alpha_\epsilon = \alpha + \epsilon 1_n$ and $\ccp(E,\calM_n)$ denotes the space of completely contractive and completely positive maps from $E$ to $\calM_n$. We will refer to the map $E \to E^\sharp : x \to (x,0)$ as the {\em canonical inclusion map}, and we will refer to the map $E^\sharp \to \C : (x,\alpha) \to \alpha$ as the {\em projection onto the scalar summand}. 
\end{definition}

The next result is contained in \cite{Wer2002}*{Section 4} (see also \cite{CS2020}*{Proposition 2.12} and \cite{CS2020}*{Lemma 2.13}).

\begin{theorem} \label{thm:partial-unitization}
Let $E$ be a matrix ordered operator space.
\begin{enumerate}
\item The partial unitization $E^\sharp$ is a unital operator system.
\item Let $\iota : E \to E^\sharp$ denote the canonical inclusion map and let $\tau : E^\sharp \to \C$ denote the projection onto the scalar summand. Then $\iota$ is completely contractive and completely positive and $\tau$ is unital and positive, and the following sequence is split exact:
\[
\begin{tikzcd}
0 \arrow{r} & E \arrow{r}{\iota} & E^\sharp \arrow{r}{\tau} & \C \arrow{r} & 0\;.
\end{tikzcd}
\]
\item Let $F$ be a matrix ordered operator space and let $\phi : E \to F$ be a completely contractive and completely positive map. Then the unitization $\phi^\sharp : E^\sharp \to F^\sharp$ defined by $\phi^\sharp((x,\alpha)) = (\phi(x),\alpha)$ for $(x,\alpha) \in E^\sharp$ is unital and completely positive. Furthermore, if $\phi$ is a completely isometric complete order isomorphism then $\phi^\sharp$ is a unital complete order isomorphism.
\end{enumerate}
\end{theorem}

\begin{remark}
Note that $E \ne E^\sharp$, even if $E$ is already unital. For a C*-algebra $A$, the partial unitization $A^\sharp$ coincides with the usual C*-algebraic unitization of $A$, and hence is a unital C*-algebra.
\end{remark}

It follows from the representation theorem of Choi and Effros \cite{CE1977} for unital operator systems that if $E$ is a matrix ordered operator space with partial unitization $E^\sharp$ and the canonical inclusion map $E \to E^\sharp$ is completely isometric, then there is a completely isometric complete order isomorphism of $E$ onto a self-adjoint subspace of bounded operators acting on a Hilbert space. Following \cite{CS2020}, this motivates the following definition. 

\begin{definition} \label{defn:generalized-operator-system}
We will say that a complete matrix ordered operator space $S$ is an {\em operator system} if the canonical inclusion map $S \to S^\sharp$ is completely isometric, in which case we will refer to $S^\sharp$ as the {\em unitization} of $S$.
\end{definition}
\begin{remark}
As in the unital case, it is not strictly necessary to assume that operator systems are complete. For an operator system $S$, we will identify $S$ with its image in $S^\sharp$ under the canonical inclusion map.	 In particular, if $T$ is an operator system and $\phi : S \to T$ is completely contractive and completely positive, then we will view the unitization $\phi^\sharp : S^\sharp \to T^\sharp$ as an extension of $\phi$. 
\end{remark}

\begin{remark}
If $S$ is a unital operator system, then it follows from \cite{Wer2002}*{Lemma 4.9} that the identity map on $S$ factors through the canonical inclusion map $S^\sharp$. In particular, this implies that the canonical inclusion map is completely isometric, so $S$ is an operator system in the sense of Definition \ref{defn:generalized-operator-system}. 
\end{remark}

\begin{remark} \label{rem:difficulty-complete-order-embeddings}
	Let $S$ and $T$ be operator systems and let $\phi : S \to T$ be a completely contractive completely positive map. If $\phi$ is a completely isometric complete order isomorphism, then Theorem \ref{thm:partial-unitization} implies that the unitization $\phi^\sharp : S^\sharp \to T^\sharp$ is a complete order isomorphism. However, if $\phi$ is merely a completely isometric complete order embedding, then it is not necessarily true that the unitization $\phi^\sharp$ is a complete order embedding (see Example \ref{ex:does-not-extend-to-complete-order-embedding}). We will need to take this into account when we define embeddings between operator systems below.
\end{remark}

In the following example, we construct operator systems $S$ and $T$ and a completely isometric complete order embedding $\phi : S \to T$ such that the unitization $S^\sharp \to T^\sharp$ is not a complete order embedding. The fundamental issue is that completely contractive completely positive maps on the image of $S$ in $T$ do not necessarily extend to completely contractive completely positive maps on $T$ (see \cite{Rus2015}*{Section 6}).

\begin{example} \label{ex:does-not-extend-to-complete-order-embedding}
Define $a,b \in \calM_2$ by
\[
a = \left[\begin{matrix} 1 & 0 \\ 0 & -1 \end{matrix}\right],\quad 
b = \left[\begin{matrix} 1 & 0 \\ 0 & -1/2 \end{matrix}\right].
\]
Let $S = \operatorname{span}\{a\}$ and $B = \operatorname{span}\{1_2,b\} \cong \C^2$. Then $S$ is a non-unital operator system and $B$ is a unital C*-algebra. Define $\phi : S \to B$ by $\phi(\alpha a) = \alpha b$ for $\alpha \in \C$. We claim that $\phi$ is a completely isometric complete order embedding, but that the unitization $\phi^\sharp : S^\sharp \to B^\sharp$ is not completely isometric.

Note that $\calM_n(S) = \operatorname{span}\{\alpha \otimes a : \alpha \in \calM_n \}$. Since
\[
\|\phi(\alpha \otimes a)\| = \|\alpha \otimes b\| = \|\alpha\| = \|\alpha \otimes a\|,
\]
$\phi$ is completely isometric. Also, $\alpha \otimes a \geq 0$ if and only if $\alpha \otimes b \geq 0$ if and only if $\alpha = 0$, so $\alpha$ is a complete order embedding.

It is not difficult to see that for $\lambda \in [-1,1]$ the map $\phi_\lambda : S \to \C$ defined by $\phi_\lambda(\alpha a) = \lambda \alpha$ for $\alpha \in \C$ is a quasistate, i.e.\ is completely contractive and completely positive. Furthermore, if $\psi : S \to \C$ is a quasistate, then $\psi = \phi_\lambda$ for some $\lambda \in [-1,1]$. Hence the set of quasistates on $S$ can be identified with $[-1,1]$.

We will see in Section \ref{prop:nc-state-space-unitization} that this implies that the state space of the unitization $S^\sharp$ is $[-1,1]$. Since $[-1,1]$ is a simplex, it follows from a classical result of Bauer that $S^\sharp = \operatorname{span} \{1_2,a\} \cong \C^2$ (see e.g.\ \cite{KS2019}).

Note that $B^\sharp \cong \C^3$. We can identify $B = \C^2$ with the first two coordinates of $\C^3$. Then $\phi^\sharp(\alpha 1_2 + \beta a) = \alpha 1_3 + \beta b$. In particular, $\phi^\sharp(\tfrac{1}{2} 1_2 + a) = \tfrac{1}{2} 1_3 + b$. Since $\tfrac{1}{2} 1_2 + a \not \geq 0$ but $\tfrac{1}{2} 1_3 + b \geq 0$, it follows that $\phi^\sharp$ is not a complete order embedding.
\end{example}

We will write $\OpSys$ for the category of operator systems with completely contractive and completely positive maps as morphisms. We will refer to completely isometric complete order isomorphisms as {\em isomorphisms}. Motivated by Remark \ref{rem:difficulty-complete-order-embeddings}, for operator systems $S$ and $T$, we will refer to a completely isometric complete order embedding $\phi : S \to T$ as an {\em embedding} if the unitization $\phi^\sharp : S^\sharp \to T^\sharp$ is an embedding in the category of unital operator systems.

Werner was able to isolate the precise obstruction to a matrix ordered operator space being an operator system in the sense of Definition \ref{defn:generalized-operator-system}. The next result is \cite{Wer2002}*{Lemma 4.8}.

\begin{theorem} \label{thm:werner-norm}
Let $E$ be a matrix ordered operator space with partial unitization $E^\sharp$. For each $n$, let $\nu_n : \calM_n(E) \to \R_{\geq 0}$ denote the map defined by
\[
\nu_n(x) = \sup_\phi \left| \phi \left(\begin{matrix} 0 & x \\ x^* & 0 \end{matrix}\right) \right|, \quad \text{for} \quad x \in \calM_n(E),
\]
where the supremum is taken over all maps $\phi \in \ccp(\calM_{2n}(E),\C)$. Then $\nu_n$ is a norm on $\calM_n(E)$. The inclusion $E \to E^\sharp$ is completely isometric if and only if for each $n$, the norm on $\calM_n(E)$ coincides with $\nu_n$.
\end{theorem}

%%%
\section*{Acknowledgements}
%%%

The authors are grateful to Ken Davidson for a number of helpful comments. The authors are also grateful to the anonymous referee for their corrections and suggestions.

%%%
\section{Pointed noncommutative convex sets} \label{sec:pointed-nc-convex-sets}
%%%

A key result from \cite{DK2019} is the dual equivalence between the category of unital operator systems and the category of compact nc convex sets. In this section we will review the definition of a compact nc convex set and introduce the definition of a pointed compact nc convex set. In Section \ref{sec:categorical-duality}, we will show that the category of operator systems is dual to the category of pointed compact nc convex sets.

%%%
\subsection{Noncommutative convex sets}
%%%

Let $E$ be an operator space. For nonzero (potentially infinite) cardinals $m$ and $n$, let $\calM_{m,n}(E)$ denote the operator space of $m \times n$ matrices over $X$ with the property that the set of finite submatrices are uniformly bounded. For brevity, we will write $\calM_n(E)$ for $\calM_{n,n}(E)$, $\calM_{m,n}$ for $\calM_{m,n}(\C)$ and $\calM_n$ for $\calM_n(\C)$. Restricting to matrices with uniformly bounded finite submatrices ensures that matrices over $E$ can be multiplied on the left and right by scalar matrices of the appropriate size. We identify $\calM_n$ with the C*-algebra of bounded operators acting on a Hilbert space $H_n$ of dimension $n$.

If $E$ is a dual operator space with distinguished predual $E_*$, then there is a natural operator space isomorphism $\calM_n(E) \cong \operatorname{CB}(E_*,\calM_n)$, where $\operatorname{CB}(E_*,\calM_n)$ denotes the space of completely bounded maps from $E_*$ to $\calM_n$. We equip $\calM_n(E)$ with the corresponding point-weak* topology.

Let $\calM(E) = \sqcup_n \calM_n(E)$, where the union is taken over all nonzero cardinal numbers $n$. Once again, for brevity, we will write $\calM$ for $\calM(\C)$. Although $\calM(E)$ is a proper class and not a set, we will only be interested in subsets, so this will not present any set-theoretic difficulties. More generally, we will consider disjoint unions over nonzero cardinal numbers $n$ of subsets of $\calM_n(E)$. For a subset $X \subseteq \calM(E)$ and a cardinal number $n$, we will write $X_n$ for the graded component $X_n = X \cap \calM_n(E)$.

\begin{definition}
Let $E$ be an operator space. An {\em nc convex set} over $E$ is a graded subset $K = \sqcup_n K_n$ with $K_n \subseteq \calM_n(E)$ that is closed under direct sums and compressions, meaning that
\begin{enumerate}
\item $\sum \alpha_i x_i \alpha_i^* \in K_n$ for every bounded family of points $x_i \in K_{n_i}$ and every family of isometries $\alpha_i \in \calM_{n,n_i}$ satisfying $\sum \alpha_i \alpha_i^* = 1_n$.
\item $\beta^* x \beta \in K_n$ for every $x \in K_n$ and every isometry $\beta \in \calM_{m,n}$.
\end{enumerate}
If $E$ is a dual operator space, so that each $\calM_n(E)$ is equipped with the weak* topology discussed above, then we will say that $K$ is {\em closed} if each $K_n$ is closed. Similarly, we will say that $K$ is {\em compact} if each $K_n$ is compact.
\end{definition}

The most important examples of compact nc convex sets are noncommutative state spaces of operator systems. The next definition is \cite{DK2019}*{Example 2.2.6}.

\begin{definition}
Let $S$ be a unital operator system. The {\em nc state space} of $S$ is the set $K = \sqcup_n K_n$ defined by $K_n = \ucp(S,\calM_n)$. Here, $\ucp(S,\calM_n)$ denotes the space of unital completely positive maps from $S$ to $\calM_n$. Elements in $K$ are referred to as {\em nc states} on $S$. 
\end{definition}

\begin{remark}	
Note that the set $K$ is nc convex and compact since each $\ucp(S,\calM_n)$ is compact.
\end{remark}

The following characterization of compact nc convex sets as sets that are closed under nc convex combinations is often useful. In particular, it makes the analogy between nc convex sets and ordinary convex sets more explicit. The next result is \cite{DK2019}*{Proposition 2.2.8}.

\begin{proposition}
Let $E$ be a dual operator space and let $K = \sqcup_n K_n$ for closed subsets $K_n \subseteq \calM_n(E)$. Then $K$ is nc convex if and only if it is closed under {\em nc convex combinations}, meaning that $\sum \alpha_i^* x_i \alpha_i \in K_n$ for every bounded family of points $x_i \in K_{n_i}$ and every family $\alpha_i \in \calM_{n_i,n}$ satisfying $\sum \alpha_i^* \alpha_i = 1_n$.
\end{proposition}

One of the most important justifications for the utility of noncommutative convexity is the fact that there is a robust notion of extreme point for which a noncommutative analogue of the Krein-Milman theorem \cite{DK2019}*{Theorem 6.4.2} holds, meaning that every compact nc convex set is generated by its extreme points.

\begin{definition}
Let $K$ be a compact nc convex set. A point $x \in K_n$ is {\em extreme} if whenever $x$ is written as a finite nc convex combination  $x = \sum \alpha_i^* x_i \alpha_i$ for $\{x_i \in K_{n_i}\}$ and nonzero $\{\alpha_i \in \calM_{n_i,n}\}$ satisfying $\sum \alpha_i^* \alpha_i = 1_n$, then each $\alpha_i$ is a positive scalar multiple of an isometry $\beta_i \in \calM_{n_i,n}$ satisfying $\beta_i^* x_i \beta_i = x$ and each $x_i$ decomposes with respect to the range of $\alpha_i$ as a direct sum $x_i = y_i \oplus z_i$ for $y_i,z_i \in K$ with $y_i$ unitarily equivalent to $x$. The set of all extreme points is $\partial K = \sqcup (\partial K)_n$.
\end{definition}

The morphism between nc convex sets are the continuous affine noncommutative maps. The next definition is \cite{DK2019}*{Definition 2.5.1}.

\begin{definition} \label{defn:nc-map}
Let $K$ and $L$ be compact nc convex sets. A map $\theta : K \to L$ is an {\em nc map} if it is graded, respects direct sums and is unitarily equivariant, meaning that
\begin{enumerate}
\item $\theta(K_n) \subseteq L_n$ for all $n$,
\item $\theta(\sum \alpha_i x_i \alpha_i^*) = \sum \alpha_i \theta(x_i) \alpha_i^*$ for every bounded family $\{x_i \in K_{n_i}\}$ and every family of isometries $\{\alpha_i \in \calM_{n_i,n}\}$ satisfying $\sum \alpha_i^* \alpha_i = 1_n$,
\item $\theta(\alpha^* x \alpha) = \alpha^* \theta(x) \alpha$ for every $x \in K_m$ and every unitary $\alpha \in \calM_n$.
\end{enumerate}
An nc map $\theta$ is {\em affine} if, in addition, it is equivariant with respect to isometries, meaning that
\begin{enumerate}
\item[(3')] $\theta(\alpha^* x \alpha) = \alpha^* \theta(x) \alpha$ for every $x \in K_m$ and every isometry $\alpha \in \calM_{m,n}$.	
\end{enumerate}
An affine nc map $\theta$ is {\em continuous} if the restriction $f|_{K_n}$ is continuous with respect to the point-strong topology on $K_n$ and $L_n$ for each $n$. It is {\em bounded} if $\|\theta\|_\infty < \infty$, where $\|\theta\|_\infty$ denotes the uniform norm $\|\theta\|_\infty = \sup_{x \in K} \|\theta(x)\|$. Finally, $\theta$ is an {\em (affine) homeomorphism} and $K$ and $L$ are {\em (affinely) homeomorphic} if $\theta$ is continuous and has a continuous (affine) nc inverse. 
\end{definition}

\begin{remark}
We will consider an appropriate notion of continuity for more general nc maps in Section \ref{sec:pointed-nc-functions}.
\end{remark}

We will write $\NCConv$ for the category of compact nc convex sets with continuous affine nc maps as morphisms. We will refer to affine nc homeomorphisms as {\em isomorphisms}, and to injective continuous affine nc maps as {\em embeddings}.

The next definition is \cite{DK2019}*{Definition 3.2.1}.

\begin{definition}
Let $K$ be a compact nc convex set. We will write $\rmA(K)$ for the unital operator system of all continuous affine nc functions from $K$ to $\calM$.
\end{definition}

\begin{remark}
The fact that $\rmA(K)$ is a unital operator system is discussed in \cite{DK2019}*{Section 3.2}.
\end{remark}

For a point $x \in K_n$, the corresponding evaluation map $\rmA(K) \to \calM_n : a \to a(x)$ is an nc state on $\rmA(K)$. Moreover, by \cite{DK2019}*{Theorem 3.2.2}, every nc state on $\rmA(K)$ is given by evaluation at some point in $K$ (we will say more about this in Section \ref{sec:categorical-duality}). It will be convenient to identify points in $K$ with the corresponding nc state on $\rmA(K)$. 

Following \cite{DK2019}*{Section 3.2}, for each $n$ we will identify the unital operator system $\calM_n(\rmA(K))$ with the space of continuous affine nc maps from $K$ to $\calM_n(\calM)$ in the obvious way.

%%%
\subsection{Pointed noncommutative convex sets}
%%%

In this section we introduce the notion of a pointed compact nc convex set, of which the most important examples will be nc quasistate spaces of operator systems. Before introducing the definition of a pointed compact nc convex set, we require the definition of a pointed continuous affine nc function.

\begin{definition}
Let $(K,z)$ be a pair consisting of a compact nc convex set $K$ and a point $z \in K_1$. We will say that a continuous affine nc function $a \in \rmA(K)$ is {\em pointed} if $f(z) = 0$. We let $\rmA(K,z) \subseteq \rmA(K)$ denote the space of pointed continuous affine nc functions on $K$.
\end{definition}

\begin{remark}
The space $\rmA(K,z)$ is a matrix ordered operator space with matrix cone $P = \sqcup P_n$ inherited from $\rmA(K)$. Specifically, for $n \in \N$, the positive cone on $P_n$ consists of the positive functions in $\calM_n(\rmA(K,z))$. Since $\rmA(K,z)$ is a closed self-adjoint subspace of the unital operator system $\rmA(K)$, it follows that $\rmA(K,z)$ is an operator system.
\end{remark}

The most important examples of pointed compact nc convex sets will be nc quasistate spaces of operator systems. The idea to utilize nc quasistate spaces in this setting was inspired by the importance of the quasistate space of a non-unital C*-algebra. 

\begin{definition} \label{defn:nc-quasistate-space}
Let $S$ be an operator system. The {\em nc quasistate space} of $S$ is the pair $(K,z)$, where $K = \sqcup_n K_n$ is defined by $K_n = \ccp(S,\calM_n)$ and $z \in K_1$ is the zero map. Here, $\ccp(S,\calM_n)$ denotes the space of completely contractive and completely positive maps from $S$ to $\calM_n$. We will refer to elements of $K$ as {\em nc quasistates} on $S$. 
\end{definition}

\begin{remark}
Note that the set $K$ is nc convex and compact since each $\ccp(S,\calM_n)$ is compact.
\end{remark}

We are now ready to introduce the definition of a pointed compact nc convex set.

\begin{definition} \label{defn:pointed-compact-nc-convex-set}
	Let $(K,z)$ be a pair consisting of a compact nc convex set $K$ and a point $z \in K_1$. We will say that $(K,z)$ is a {\em pointed compact nc convex set} if every nc quasistate on the operator system $\rmA(K,z)$ belongs to $K$, i.e.\ is evaluation at a point in $K$.
\end{definition}

\begin{remark}
Since $K$ is the nc state space of the unital operator system $\rmA(K)$, $(K,z)$ is a pointed compact nc convex set if and only if every nc quasistate on $\rmA(K,z)$ extends to an nc state on $\rmA(K)$. 	
\end{remark}

By definition, a pointed compact nc convex set is the nc quasistate space of an operator system. In Section \ref{sec:categorical-duality}, we will show that the nc quasistate space of every operator system is a pointed compact nc convex set. The proof of this fact is non-trivial. However, we are now able to give some examples. 

\begin{example} \label{ex:example-pointed-1}
Define $K = \sqcup K_n$ by
\[
K_n = \{\alpha \in (\calM_n)_h : -1_n \leq \alpha \leq 1_n \}, \quad \text{for} \quad n \in \N.
\]
Then $K$ is a compact nc convex set (see \cite{DK2019}*{Example 2.2.4}). Let $z = 0$. We will show that the pair $(K,z)$ is a pointed compact nc convex set.

The unital operator system $\rmA(K)$ is given by $\rmA(K) = \operatorname{span}\{1_{\rmA(K)},a\}$, where $a \in \rmA(K,z)$ is the coordinate function $a(\alpha) = \alpha$ for $\alpha \in K$. Hence $\rmA(K,z) = \operatorname{span}\{a\}$. In fact, $\rmA(K,z)$ and $\rmA(K)$ are isomorphic to the operator systems $S$ and $S^\sharp$ from Example \ref{ex:does-not-extend-to-complete-order-embedding}. 

If $\theta : \rmA(K,z) \to \calM_n$ is an nc quasistate, then there is a self-adjoint $\beta \in \calM_n$  with $-1_n \leq \beta \leq 1_n$ such that $\theta(\alpha a) = \alpha \beta$ for $\alpha \in \C$. Conversely, it is easy to check that every self-adjoint $\beta \in \calM_n$ with $-1_n \leq \beta \leq 1_n$ gives rise to an nc quasistate on $\rmA(K,z)$ of this form. Hence the nc quasistate space of $\rmA(K,z)$ is $K$. Therefore, $(K,z)$ is a pointed compact nc convex set. 

Note that $\rmA(K,z)^\sharp = \rmA(K)$. In Corollary \ref{cor:pointed-conditions}, we will show that this property characterizes pointed compact nc convex sets.
\end{example}

The next example shows that not every pair $(K,z)$ consisting of a compact nc convex set and a point $z \in K_1$ is a pointed compact nc convex set.

\begin{example} \label{ex:first-example-pointed-2}
Define $K = \sqcup K_n$ by
\[
K_n = \{\alpha \in (\calM_n)_h : -\tfrac{1}{2} 1_n \leq \alpha \leq 1_n \}, \quad \text{for} \quad n \in \N.
\]
Then as in Example \ref{ex:example-pointed-1}, $K$ is a compact nc convex set. Let $z = 0$. We will show that the pair $(K,z)$ is not a pointed compact nc convex set.

The unital operator system $\rmA(K)$ is given by $\rmA(K) = \operatorname{span}\{1_{\rmA(K)},b\}$, where $b \in \rmA(K,z)$ is the coordinate function $b(\alpha) = \alpha$ for $\alpha \in K$. Hence $\rmA(K,z) = \operatorname{span}\{b\}$. In fact, $\rmA(K)$ is isomorphic to the C*-algebra $B$ from Example \ref{ex:does-not-extend-to-complete-order-embedding}.

Define $\theta : \rmA(K,z) \to \C$ by $\theta(\alpha b) = -\alpha$. Since $\rmA(K,z)$ does not contain any positive elements, the matrix cone of $\rmA(K,z)$ is zero, so it is easy to check that $\theta$ is an nc quasistate. However, $\theta$ does not extend to an nc state on $\rmA(K)$ since $\tfrac{1}{2} 1_{\rmA(K)} + b \geq 0$, while $\tfrac{1}{2} + \theta(b) = -\tfrac{1}{2} \not \geq 0$. Hence $\theta$ does not belong to $K$ and $(K,z)$ is not a pointed compact nc convex set.
\end{example}

We will now establish a geometric characterization of pointed compact nc convex sets.

\begin{proposition} \label{prop:geometric-condition-pointed}
	Let $(K,z)$ be a pair consisting of a compact nc convex set $K$ and a point $z \in K_1$. Then $(K,z)$ is pointed if and only if whenever self-adjoint $a \in \calM_m(\rmA(K,z))$ satisfies $a(x) \leq 1_m \otimes 1_n$ for all $x \in K_n$, then $\theta_m(a) \leq 1_m \otimes 1_p$ for all nc quasistates $\theta : \rmA(K,z) \to \calM_p$.
\end{proposition}

\begin{proof}
If $(K,z)$ is pointed, then every nc quasistate on $\rmA(K,z)$ belongs to $K$, so the condition trivially holds. Conversely, suppose that $(K,z)$ is not pointed. Then there is an nc quasistate $\theta : \rmA(K,z) \to \calM_n$ such that $\theta \notin K$. Identifying $K$ with its image in $\calM(\rmA(K,z)^*)$ and viewing $\theta$ as a point in $\calM_n(\rmA(K,z)^*)$, the nc separation theorem \cite{DK2019}*{Theorem 2.4.1} implies there is a self-adjoint element $a \in \calM_n(\rmA(K,z))$ such that $\theta(a) \not \leq 1_n \otimes 1_n$ but $a(x) \leq 1_n \otimes 1_p$ for all $x \in K_p$. 
\end{proof}

Example \ref{ex:first-example-pointed-2} is a single instance of a general class of examples.

\begin{corollary} \label{cor:condition-zero-matrix-cone}
		Let $(K,z)$ be a pair consisting of a compact nc convex set and a point $z \in K_1$ such that the matrix cone for $\rmA(K,z)$ is zero. Then $(K,z)$ is pointed if and only if whenever $x \in K_n$ satisfies $\alpha z^{(n)} + (1-\alpha) x \in K_n$ for $0 < \alpha < 1$, then $\alpha z^{(n)} - (1 - \alpha) x \in K_n$. Here $z^{(n)} \in K_n$ denotes the direct sum of $n$ copies of $z$.
\end{corollary}

\begin{proof}
Suppose that $(K,z)$ is pointed and $x \in K_n$ satisfies $\alpha z^{(n)} + (1-\alpha) x \in K_n$ for $0 < \alpha < 1$. Then since the positive cone of $\rmA(K,z)$ is zero, the map $\theta : \rmA(K,z) \to \calM_n$ defined by $\theta(a) = \alpha a(z^{(n)}) - (1-\alpha) a(x) = -(1-\alpha) a(x)$ for $a \in \rmA(K,z)$ is an nc quasistate. Hence $\theta$ is given by evaluation at a point in $K$ which must be $\alpha z^{(n)} - (1-\alpha) x$. Hence $\alpha z^{(n)} - (1-\alpha) x \in K$. 

Conversely, suppose that whenever $x \in K_n$ satisfies $\alpha z^{(n)} + (1-\alpha) x \in K_n$ for $0 < \alpha < 1$, then $\alpha z^{(n)} - (1 - \alpha) x \in K_n$. If self-adoint $a \in \calM_m(\rmA(K,z))$ satisfies $a(x) \leq 1_m \otimes 1_n$ for all $x \in K_n$, then for $0 < \alpha < 1$,
\[
(1-\alpha) a(x) = a(\alpha z^{(n)} + (1-\alpha) x) \leq 1_m \otimes 1_n
\]
and 
\[
-(1-\alpha) a(x) = a(\alpha z^{(n)} - (1-\alpha) x) \leq 1_m \otimes 1_n.
\]
Then taking $\alpha \to 0$ implies
\[
-1_m \otimes 1_n \leq a(x) \leq 1_m \otimes 1_n.
\]
Hence $\|a\|_\infty \leq 1$. It follows that if $\theta : \rmA(K,z) \to \calM_n$ is an nc quasistate on $\rmA(K)$, then $\theta(a) \leq 1_m \otimes 1_n$. Therefore, by Proposition \ref{prop:geometric-condition-pointed}, $(K,z)$ is pointed.
\end{proof}

\begin{example}
Let $K$ denote the nc state space of $\calM_2$ and let $z = \operatorname{Tr}$, where $\operatorname{Tr} \in K_1$ denotes the normalized trace. Then identifying $\calM_2$ with $\rmA(K)$,
\[
\rmA(K,z) = \left\{ \left[\begin{matrix} \alpha & \beta \\ \gamma & -\alpha \end{matrix}\right] : \alpha,\beta,\gamma \in \C
  \right\}.
\]
The matrix cone of $\rmA(K,z)$ is clearly zero. Define $\theta : \calM_2 \to \C$ by
\[
\theta \left( \left[\begin{matrix} \alpha & \beta \\ \gamma & \delta \end{matrix}\right]\right) = \alpha \quad \text{for} \quad  \left[\begin{matrix} \alpha & \beta \\ \gamma & \delta \end{matrix}\right] \in \calM_2.
\]
Then $\frac{1}{2} \operatorname{Tr} + \frac{1}{2} \theta \in K_1$. But $\frac{1}{2} \operatorname{Tr} - \frac{1}{2} \theta \notin K_1$. Hence by Corollary \ref{cor:condition-zero-matrix-cone}, the pair $(K,z)$ is not pointed.
\end{example}

We now define the category of pointed compact nc convex sets. 

\begin{definition}
Let $(K,z)$ and $(L,w)$ be pointed nc convex sets. We will say that an affine nc map $\theta : K \to L$ is {\em pointed} if $\theta(z) = w$. We will say that $(K,z)$ and $(L,w)$ are {\em pointedly affinely homeomorphic} if there is a pointed affine homeomorphism from $(K,z)$ to $(L,w)$. 
\end{definition}

We will write $\PoNCConv$ for the category of pointed compact nc convex sets with pointed continuous affine nc maps as morphisms. We will refer to pointed affine nc homeomorphisms as {\em isomorphisms}, and to pointed injective continuous affine nc maps as {\em embeddings}.

%%%
\section{Categorical duality} \label{sec:categorical-duality}
%%%

In this section we will prove the dual equivalence between the category $\OpSys$ of operator systems and the category $\PoNCConv$ of pointed compact nc convex sets. We begin by reviewing the details of the dual equivalence between the category $\UnOpSys$ of unital operator systems and the category $\NCConv$ of compact nc convex sets from \cite{DK2019}.

%%%
\subsection{Categorical duality for unital operator systems}
%%%

The dual equivalence between the category of unital operator systems and the category of compact nc convex sets was developed in \cite{DK2019}*{Section 3}. It is closely related to a similar dual equivalence established by Webster and Winkler \cite{WW1999}.

The next result combines \cite{DK2019}*{Theorem 3.2.2} and \cite{DK2019}*{Theorem 3.2.3}. 

\begin{theorem} \label{thm:unital-dual-equivalence}
Let $K$ be a compact nc convex set. The nc state space of the unital operator system $\rmA(K)$ is isomorphic to $K$. For a unital operator system $S$ with nc state space $K$, the map $S \to \rmA(K) : s \to \hat{s}$ defined by
\[
\hat{s}(x) = x(s) \quad \text{for} \quad s \in S,\ x \in K
\]
is an isomorphism.
\end{theorem}

The dual equivalence between the category $\UnOpSys$ of unital operator systems and the category $\NCConv$ of compact nc convex sets follows from Theorem \ref{thm:unital-dual-equivalence}. The contravariant functor $\UnOpSys \to \NCConv$ is defined in the following way:
\begin{enumerate}
\item A unital operator system $S$ is mapped to its nc state space.
\item For unital operator systems $S$ and $T$ with nc state spaces $K$ and $L$ respectively, a morphism $\phi : S \to T$ is mapped to the morphism $\phi^d : L \to K$ defined by
\[
\phi^d(y)(a) = \phi(a)(y), \quad \text{for} \quad y \in L \text{ and } a \in \rmA(K).
\]
\end{enumerate}
The inverse functor $\NCConv \to \UnOpSys$ is defined in the following way:
\begin{enumerate}
	\item A compact nc convex set $K$ is mapped to the unital operator system $\rmA(K)$.
	\item If $K$ and $L$ are compact nc convex sets and $\psi : L \to K$ is a morphism, then the corresponding morphism $\psi^d : \rmA(K) \to \rmA(L)$ is defined by
\[
\psi^d(a)(y) = a(\psi(y)), \quad \text{for} \quad a \in \rmA(K) \text{ and } y \in L.
\]
\end{enumerate}
The next result summarizes this discussion. It is \cite{DK2019}*{Theorem 3.2.5}.

\begin{theorem} \label{thm:unital-categorical-equivalence}
The contravariant functors $\UnOpSys \to \NCConv$ and $\NCConv \to \UnOpSys$ defined above are inverses. Hence the categories $\UnOpSys$ and $\NCConv$ are dually equivalent.
\end{theorem}

We will make use of the following result in the next section.

\begin{proposition} \label{prop:unital-complete-order-embedding}
Let $K$ and $L$ be compact nc convex sets. Let $\phi : \rmA(K) \to \rmA(L)$ be a unital completely positive map and let $\phi^d : L \to K$ denote the continuous affine map obtained by applying Theorem \ref{thm:unital-categorical-equivalence} to $\phi$. Then $\phi$ is completely isometric if and only if $\phi^d$ is surjective.
\end{proposition}

\begin{proof}
If $\phi^d$ is surjective, then for $a \in \calM_n(\rmA(K))$,
\[
\|\phi(a)\|_\infty = \sup_{y \in L} \|\phi(a)(y)\|_\infty = \sup_{y \in L} \|a(\phi^d(y))\|_\infty  = \sup_{x \in K} \|a(x)\| = \|a\|_\infty.
\]
Hence $\phi$ is completely isometric.

Conversely, suppose that $\phi$ is completely isometric. Let $S = \phi(\rmA(K))$. Then $S$ is a unital operator system. Let $M$ denote the nc state space of $S$, so that $S$ is isomorphic to $\rmA(K)$. It follows from Arveson's extension theorem that the restriction map $r : L \to M$ is surjective. Let $\psi : M \to K$ denote the continuous affine nc map obtained by restricting the range of $\phi$ to $S$ and applying Theorem \ref{thm:unital-categorical-equivalence}. Then $\phi^d = \psi \circ r$. Theorem \ref{thm:unital-categorical-equivalence} implies that $\psi$ is an affine homeomorphism. Since $r$ is surjective, it follows that $\phi^d$ is surjective.
\end{proof}

%%%
\subsection{Categorical duality for operator systems}
%%%

Let $(K,z)$ be a pair consisting of a compact nc convex set $K$ and a point $z \in K_1$. Observe that for a point $x \in K$, viewed as a unital completely positive map on $\rmA(K)$, the restriction $x|_{\rmA(K,z)}$ is an nc quasistate. For brevity, it will be convenient to simultaneously view points in $K$ as nc states on $\rmA(K)$ and nc quasistates on $\rmA(K,z)$. We will take care to ensure that this does not cause any confusion.

%The proof of the next result foreshadows some of the results that we will consider in Section \ref{sec:categorical-duality}.

If $(K,z)$ is the nc quasistate space of an operator system $S$, then it follows as in \cite{Wer2002}*{Lemma 4.9} that the extension $x^\sharp : S^\sharp \to \calM_n$ defined by $x^\sharp(s,\alpha) = x(s) + \alpha 1_n$ is unital and completely positive, and hence is an nc state on $S^\sharp$. Moreover, it is the unique extension of $x$ to an nc state on $S^\sharp$ with range in $\calM_n$. Here we have identified $S^\sharp$ with $S \oplus \C$ as in Definition \ref{defn:partial-unitization}. Note that $S = \ker z^\sharp$. 

% For $x \in K$, it follows from \cite{Wer2002}*{Lemma 4.9} that the extension $x^\sharp : S^\sharp \to \caM_n$ defined by $x^\sharp(s,\alpha) = x(s) + \alpha 1_n$ is unital and completely positive, and hence defines an nc state on $S^\sharp$. Here we have identified $S^\sharp$ with $S \oplus \C$ as in Definition \ref{defn:partial-unitization}.

\begin{proposition} \label{prop:nc-state-space-unitization}
Let $S$ be an operator system with nc quasistate space $(K,z)$ and let $L$ denote the nc state space of the unitization $S^\sharp$. For an nc quasistate $x \in K$, let $x^\sharp \in L$ be the nc state defined as above. Then the map $K \to L : x \to x^\sharp$ is an affine homeomorphism with inverse given by the restriction map $L \to K : y \to y|_S$. Hence $S^\sharp$ is isomorphic to $\rmA(K)$. 
\end{proposition}

\begin{proof}
For $x \in K$, we have already observed that $x^\sharp \in L$. On the other hand, for $y \in L$, the restriction $y|_S$ is completely contractive and completely positive, so $y|_S \in K$. Then by uniqueness, $(y|_S)^\sharp = y$. It follows that the map $K \to L : x \to x^\sharp$ is a bijection with inverse given by the restriction map.

It is clear that the restriction map from $L$ to $K$ is continuous and affine. From above, the restriction to each $L_n$ is a continuous bijection onto $K_n$. Since $L_n$ is compact, it follows that this restriction is a homeomorphism. Hence the restriction map is a homeomorphism.

The fact that $S^\sharp$ is isomorphic to $\rmA(K)$ now follows from Theorem~\ref{thm:unital-dual-equivalence}. 
\end{proof}

\begin{theorem} \label{thm:operator-system-isomorphism}
Let $S$ be an operator system with nc quasistate space $(K,z)$. Then $S$ is isomorphic to $\rmA(K,z)$. 
\end{theorem}

\begin{proof}
By Proposition \ref{prop:nc-state-space-unitization}, we can identify the nc state space of the unitization $S^\sharp$ with $K$. Let $\phi^\sharp : S^\sharp \to A(K)$ denote the isomorphism from Theorem \ref{thm:unital-dual-equivalence}. Then for $s \in S^\sharp$, $\phi^\sharp(s) = \hat{s}$, where $\hat{s} : K \to \calM$ is the affine nc function defined by $\hat{s}(x) = x^\sharp(s)$ for $x \in K$. In particular, for $s \in S$, $\phi^\sharp(s)(z) = \hat{s}(z) = z^\sharp(s) = 0$, so $\phi^\sharp(S) \subseteq \rmA(K,z)$. Hence restricting $\phi^\sharp$ to $S$, we obtain a map $\phi : S \to \rmA(K,z)$. Since $\phi^\sharp$ is an isomorphism, $\phi$ is completely positive and completely isometric. It remains to show that $\phi$ is a surjective complete order isomorphism.

To see that $\phi$ is surjective, choose $a \in \rmA(K,z)$. By the surjectivity of $\phi^\sharp$, there is $s \in S^\sharp$ such that $\phi^\sharp(s) = a$. Then $0 = a(z) = \hat{s}(z) = z^\sharp(s)$. Hence $s \in S$, and we conclude that $\phi$ is surjective.

To see that $\phi$ is a complete order isomorphism, let $P = \sqcup P_n$ and $Q = \sqcup Q_n$ denote the matrix cones of $S$ and $\rmA(K,z)$ respectively. If $\phi$ is not a complete order isomorphism, then there is $s \in \calM_n(S)$ such that $s \notin P_n$ but $\phi(s) \in Q_n$. Suppose that this is the case. We will apply a separation argument to obtain a contradiction.

Identify $S$ with its image under the canonical embedding into its bidual $S^{**}$ and define $M \subseteq \calM(S^{**})$ by $M = \overline{P} = \sqcup \overline{P_m}$, where the closure is taken with respect to the weak* topology. Since $P$ is nc convex, $M$ is nc convex. Hence $M$ is a weak* closed nc convex set. Furthermore, since $P_n$ is convex and uniformly closed, it is weakly closed, implying $s \notin M$. Therefore, by the nc separation theorem \cite{DK2019}*{Theorem 2.4.1} there is a self-adjoint normal completely bounded linear map $\psi : S^{**} \to \calM_n$ such that $\psi(s) \not \geq -1_n \otimes 1_n$ but $\psi(t) \geq -1_n \otimes 1_p$ for all $t \in M_p$. 

Since $\psi$ is normal, it can be identified with the unique normal extension of a map $\psi : S \to \calM_n$ satisfying $\psi(s) \not \geq -1_n \otimes 1_n$ but $\psi(t) \geq -1_n \otimes 1_p$ for all $t \in P_p$. Then in particular, $\psi(s) \not \geq 0$. However, since $P$ is closed under multiplication by positive scalars, for $t \in P_p$ and $\alpha > 0$, $\psi(t) \geq -\alpha^{-1} 1_n \otimes 1_p$. Taking $\alpha \to \infty$ implies $\psi(t) \geq 0$. Hence $\psi \geq 0$. Multiplying $\psi$ by a sufficiently small positive scalar, we obtain a quasistate $x \in K$ such that $x(s) \not \geq 0$. But then $\hat{s}(x) = x(s) \not \geq 0$, so $\phi(s) = \hat{s} \not \geq 0$, contradicting the assumption that $\phi(s) \in Q_n$. 
\end{proof}

\begin{corollary}
	Let $S$ be an operator system with nc quasistate space $(K,z)$. Then $(K,z)$ is a pointed compact nc convex set.
\end{corollary}

\begin{proof}
By Theorem \ref{thm:operator-system-isomorphism}, we can identify $S$ with the operator system $\rmA(K,z)$, and by definition, every nc quasistate on $\rmA(K,z)$ belongs to $K$.
\end{proof}

\begin{corollary} \label{cor:pointed-conditions}
Let $(K,z)$ be a pair consisting of a compact nc convex set $K$ and a point $z \in K_1$. The following are equivalent:
\begin{enumerate}
	\item The pair $(K,z)$ is a pointed compact nc convex set.
	\item The nc quasistate space of the operator system $\rmA(K,z)$ is $(K,z)$.
	\item The operator system $\rmA(K,z)$ satisfies $\rmA(K,z)^\sharp = \rmA(K)$. 
\end{enumerate}
\end{corollary}
\begin{proof}
(1) $\Rightarrow$ (2) If $(K,z)$ is a pointed compact nc convex set then by definition every nc quasistate on $\rmA(K,z)$ belongs to $K$. Since every point in $K$ is an nc quasistate on $\rmA(K,z)$, it follows that the nc quasistate space of $\rmA(K,z)$ is $(K,z)$.

(2) $\Rightarrow$ (3) If the nc quasistate space of $\rmA(K,z)$ is $(K,z)$, then Proposition \ref{prop:nc-state-space-unitization} implies that the nc state space of $\rmA(K,z)^\sharp$ is $K$. It follows from Theorem \ref{thm:unital-categorical-equivalence} that $\rmA(K,z)^\sharp = \rmA(K)$. 

(3) $\Rightarrow$ (1) If $\rmA(K,z)^\sharp = \rmA(K)$, then since every nc quasistate on $\rmA(K,z)$ extends to an nc state on $\rmA(K,z)^\sharp$, and since $K$ is nc state space of $\rmA(K)$, it follows that every nc quasistate of $\rmA(K,z)$ belongs to $K$. Hence $(K,z)$ is a pointed compact nc convex set.
\end{proof}

The next result follows immediately from Theorem \ref{thm:operator-system-isomorphism} and Corollary \ref{cor:pointed-conditions}. It is an analogue of the representation theorem \cite{DK2019}*{Theorem 3.2.3}.

% todo

\begin{theorem} \label{thm:duality-theorem}
Let $S$ be an operator system with nc quasistate space $(K,z)$. The map $S^\sharp \to \rmA(K) : s \to \hat{s}$ defined by
\[
\hat{s}(x) = x^\sharp(s) \quad \text{for} \quad x \in K,
\]
is a unital complete order isomorphism that restricts to a completely isometric complete order isomorphism from $S$ to $\rmA(K,z)$. Hence $S$ is isomorphic to $\rmA(K,z)$. 
\end{theorem}

Theorem \ref{thm:operator-system-isomorphism} and Corollary \ref{cor:pointed-conditions} imply the dual equivalence of the category $\OpSys$ of operator systems and the category $\PoNCConv$ of pointed compact nc convex sets. The contravariant functor $\OpSys \to \PoNCConv$ is defined in the following way:

\begin{enumerate}
	\item An operator system $S$ is mapped to its nc quasistate space.
	\item For operator systems $S$ and $T$ with nc quasistate spaces $(K,z)$ and $(L,w)$ respectively, a morphism $\phi : S \to T$ is mapped to the morphism $\phi^d : L \to K$ defined by
\[
\phi^d(y)(a) = \phi(a)(y), \quad \text{for} \quad y \in L \text{ and } a \in \rmA(K,z).
\]
\end{enumerate}
The inverse functor $\PoNCConv \to \OpSys$ is defined in the following way:
\begin{enumerate}
	\item A pointed compact nc convex set $(K,z)$ is mapped to the operator system $\rmA(K,z)$.
	\item If $(K,z)$ and $(L,w)$ are compact nc convex sets and $\psi : L \to K$ is a morphism, then the corresponding morphism $\psi^d : \rmA(K,z) \to \rmA(L,w)$ is defined by
\[
\psi^d(a)(y) = a(\psi(y)), \quad \text{for} \quad a \in \rmA(K,z) \text{ and } y \in L.
\]
\end{enumerate}
The next result summarizes this discussion.

\begin{theorem} \label{thm:categorical-equivalence}
The contravariant functors $\OpSys \to \PoNCConv$ and $\PoNCConv \to \OpSys$ defined above are inverses. Hence the categories $\OpSys$ and $\PoNCConv$ are dually equivalent.
\end{theorem}

The next result characterizing isomorphic operator systems is an analogue of \cite{DK2019}*{Corollary 3.2.6}. It follows immediately from Theorem \ref{thm:categorical-equivalence}.

\begin{corollary} \label{cor:isomorphic-iff-affinely-homeomorphic}
Let $(K,z)$ and $(L,w)$ be compact pointed nc convex sets. Then $\rmA(K,z)$ and $\rmA(L,w)$ are isomorphic if and only if $(K,z)$ and $(L,w)$ are pointedly affinely homeomorphic. Hence two operator systems are isomorphic if and only if their nc quasistate spaces are pointedly affinely homeomorphic.
\end{corollary}

We saw in Example \ref{ex:does-not-extend-to-complete-order-embedding} that if $S$ and $T$ are operator systems and $\phi : S \to T$ is a completely contractive complete order embedding, then it is not necessarily true that the unitization $\phi^\sharp : S^\sharp \to T^\sharp$ is completely isometric. In other words, $\phi$ is not necessarily an embedding. However, we can now state necessary and sufficient conditions for $\phi$ to be an embedding.

The following result follows immediately from Theorem \ref{thm:unital-categorical-equivalence}, Theorem \ref{thm:categorical-equivalence} and the discussion preceding the statements of these results.

\begin{lemma} \label{lem:double-dual-map}
Let $(K,z)$ and $(L,w)$ be pointed compact nc convex sets and let $\phi : \rmA(K,z) \to \rmA(L,w)$ be a completely contractive and completely positive map. Let $\phi^d : L \to K$ denote the corresponding continuous affine map defined as in Theorem \ref{thm:categorical-equivalence}. Then $\phi^d$ coincides with the continuous affine map obtained by applying Theorem \ref{thm:unital-categorical-equivalence} to the unitization $\phi^\sharp : \rmA(K) \to \rmA(L)$.
\end{lemma}

\begin{corollary}
Let $(K,z)$ and $(L,w)$ be pointed compact nc convex sets. Let $\phi : \rmA(K,z) \to \rmA(L,w)$ be a completely contractive and completely positive map and let $\phi^d : L \to K$ denote the pointed continuous affine map given by applying Theorem \ref{thm:categorical-equivalence} to $\phi$. Then $\phi$ is an embedding if and only if $\phi^d$ is surjective.
\end{corollary}

\begin{proof}
	By Lemma \ref{lem:double-dual-map}, the map $\phi^d$ coincides with the map obtained by applying Theorem \ref{thm:unital-categorical-equivalence} to the unitization $\phi^\sharp : \rmA(K) \to \rmA(L)$. By Proposition \ref{prop:unital-complete-order-embedding}, $\phi^\sharp$ is completely isometric if and only if $\phi^d$ is surjective.
\end{proof}

%%%
\section{Pointed noncommutative functions} \label{sec:pointed-nc-functions}
%%%

%%%
\subsection{Noncommutative functions}
%%%

In order to define a more general notion of continuous nc function, it is necessary to introduce the {\em point-strong topology} on a compact nc convex set $K$. This is the weakest topology on each $K_n$ making the maps $K_n \to \C : x \to \xi^* a(x) \eta$ and $K_n \to \C : x \to a(x)^* \xi$ continuous for all $a \in \rmA(K)$ and all vectors $\xi,\eta \in H_n$.

The following definition is essentially \cite{DK2019}*{Definition 4.2.1}.

\begin{definition}
Let $K$ be a compact nc convex set. An {\em nc function} on $K$ is an nc map $f : K \to \calM$ in the sense of Definition \ref{defn:nc-map}. An nc function $f$ is {\em continuous} if it is continuous with respect to the point-strong* topology on $K$ from above. We will write $\rmB(K)$ and $\rmC(K)$ for the unital C*-algebras of bounded and continuous nc functions on $K$ respectively.
\end{definition}

\begin{remark}
It is clear that $\rmA(K) \subseteq \rmC(K) \subseteq \rmB(K)$. The product on $\rmB(K)$ is the pointwise product, meaning that for $f,g \in \rmB(K)$ and $x \in K$, $(fg)(x) = f(x)g(x)$. The adjoint is defined by $f^*(x) = f(x)^*$ for $f \in \rmB(K)$ and $x \in K$. By \cite{DK2019}*{Theorem 4.4.3}, $\rmC(K) = \ca(\rmA(K))$. We will say more about the C*-algebra $\rmC(K)$ in Section \ref{sec:min-max-c-star-covers}. 
\end{remark}

For $x \in K_n$, we will write $\delta_x : \rmB(K) \to \calM_n$ for the point evaluation *-homomorphism defined by $\delta_x(f) = f(x)$ for $f \in \rmB(K)$. This is a noncommutative analogue of an evaluation functional, since for $f \in \rmC(K)$, $\delta_x(f) = f(x)$. 

Elements in the enveloping von Neumann algebra $\rmC(K)^{**}$ can naturally be identified with bounded nc functions on $K$. Specifically, for $x \in  K_n$, it follows from the universal property of $\rmC(K)^{**}$ as the enveloping von Neumann algebra of $\rmC(K)$ that the *-homomorphism $\delta_x : \rmC(K) \to \calM_n$ has a unique extension to a normal *-homomorphism $\delta_x^{**} : \rmC(K)^{**} \to \calM_n$. For $f \in \rmC(K)^{**}$, the function $\tilde{f} : K \to \calM$ defined by $\tilde{f}(x) = \delta_x^{**}(f)$ for $x \in K$ is a bounded nc function and hence belongs to $\rmB(K)$. In fact, much more can be said. 

The following result is contained in \cite{DK2019}*{Theorem 4.4.3} and \cite{DK2019}*{Corollary 4.4.4}. 

\begin{theorem} \label{thm:unital-enveloping-von-neumann}
Let $K$ be a compact nc convex set. The map $\sigma : \rmC(K)^{**} \to \rmB(K)$ defined as above is a normal *-isomorphism that restricts to a normal unital complete order isomorphism from $\rmA(K)^{**}$ onto the unital operator system $\rmA_b(K)$ of bounded affine nc functions.
\end{theorem}

%%%
\subsection{Pointed noncommutative functions}
%%%

\begin{definition}
Let $(K,z)$ be a pointed compact nc convex set. We will say that an nc function $f : K \to \calM$ is {\em pointed} if $f(z) = 0$. We let $\rmB(K,z)$ denote the space of pointed bounded nc functions on $K$. Similarly, we let $\rmC(K,z) = \rmC(K) \cap \rmB(K,z)$ denote the space of pointed continuous nc functions on $K$.
\end{definition}

\begin{remark} \label{rem:pointed-functions-ideals}
It is clear that $\rmB(K,z)$ is a closed two-sided ideal of $\rmB(K)$ and that $\rmC(K,z)$ is a closed two-sided ideal of $\rmC(K)$. In particular, $\rmB(K,z)$ and $\rmC(K,z)$ are C*-algebras. Furthermore, it follows from the identification $\rmC(K)^{**} = \rmB(K)$ that the representation $\delta_z$ is normal on $\rmB(K)$. Hence $\rmB(K,z)$ is a weak*-closed ideal of $\rmB(K)$.
\end{remark}

\begin{proposition} \label{prop:unitizations}
Let $(K,z)$ be a pointed compact nc convex set. Then $\rmC(K,z)^\sharp = \rmC(K)$ and $\rmC(K,z) = \ca(\rmA(K,z))$.
\end{proposition}

\begin{proof}
By Corollary \ref{cor:pointed-conditions}, $\rmA(K,z)^\sharp = \rmA(K)$. Hence $\rmA(K) = \rmA(K,z) + \C 1_{\rmA(K)}$. Since $\rmC(K) = \ca(\rmA(K))$, it follows that $\rmC(K) = \rmC(K,z) + \C$. Hence $\rmC(K,z) = \ca(\rmA(K,z))$.

To see that $\rmC(K,z)^\sharp = \rmC(K)$, it suffices to show that for any *-homomorphism $\pi : \rmC(K,z) \to \calM_n$, there is a unital *-homomorphism $\tilde{\pi} : \rmC(K) \to \calM_n$ extending $\pi$. The restriction $\pi|_{\rmA(K,z)}$ is an nc quasistate, so by the assumption that $(K,z)$ is pointed, it is given by evaluation at a point $x \in K_n$. Then the unital *-homomorphism $\delta_x : \rmC(K) \to \calM_n$ extends $\pi|_{\rmA(K,z)}$. Since $\rmA(K,z)$ generates $\rmC(K,z)$, it follows that $\delta_x|_{\rmC(K,z)} = \pi$.
\end{proof}

The next result follows from restricting the *-isomorphism in the statement of Theorem \ref{thm:unital-enveloping-von-neumann}.

\begin{theorem}
Let $(K,z)$ be a compact pointed nc convex set. Then the map $\rmC(K,z)^{**} \to \rmB(K,z) : f \to \tilde{f}$ defined by
\[
\tilde{f}(x) = \delta_x^{**}(f) \quad \text{for} \quad f \in \rmC(K,z)^{**} \text{ and } x \in K,
\]
is a normal *-isomorphism of von Neumann algebras that restricts to a normal completely isometric complete order isomorphism from $\rmA(K,z)^{**}$ onto the operator system $\rmA_b(K,z)$ of pointed bounded affine nc functions.
\end{theorem}

%%%
\section{Minimal and maximal C*-covers} \label{sec:min-max-c-star-covers}
%%%

The deepest results in \cite{DK2019} arise from the interplay between unital operator systems of continuous affine nc functions on compact nc convex sets and unital C*-covers of nc functions on the sets. Connes and van Suijlekom \cite{CS2020} introduced an analogous notion of C*-cover for operator systems. In this section we will review the notion of a unital C*-cover of a unital operator system before considering the more general notion of a C*-cover of an operator system. 

%%%
\subsection{Minimal and maximal unital C*-covers} \label{sec:minimal-maximal-unital-c-star-covers}
%%%

Let $S$ be a unital operator system.
\begin{enumerate}
\item A pair $(A,\iota)$ consisting of a unital C*-algebra $A$ and an embedding $\iota : S \to A$ is a {\em unital C*-cover} of $S$ if $A = \ca(\iota(S))$.
\item If $(A',\iota')$ is another unital C*-cover of $S$, then we will say that $(A,\iota)$ and $(A,\iota')$ are {\em equivalent} if there is a unital *-isomorphism $\pi : A \to A'$ such that $\pi \circ \iota = \iota'$.
\item We will say that a unital C*-cover $(A,\iota)$ of $S$ is {\em maximal} if for any unital C*-cover $(B,\phi)$ of $S$, there is a surjective unital *-homomorphism $\sigma : A \to B$ such that $\phi = \sigma \circ \iota$.
% https://q.uiver.app/?q=WzAsMyxbMCwwLCJTIl0sWzEsMCwiQT1cXG1hdGhybXtDfV4qKFxcaW90YShTKSkiXSxbMSwxLCJCPVxcbWF0aHJte0N9XiooXFxwaGkoUykpIl0sWzAsMSwiXFxpb3RhIiwwLHsic3R5bGUiOnsidGFpbCI6eyJuYW1lIjoiaG9vayIsInNpZGUiOiJ0b3AifX19XSxbMCwyLCJcXHBoaSIsMix7InN0eWxlIjp7InRhaWwiOnsibmFtZSI6Imhvb2siLCJzaWRlIjoidG9wIn19fV0sWzEsMiwiXFxzaWdtYSIsMCx7InN0eWxlIjp7ImJvZHkiOnsibmFtZSI6ImRhc2hlZCJ9LCJoZWFkIjp7Im5hbWUiOiJlcGkifX19XV0=
\[\begin{tikzcd}
	{S} & {A=\mathrm{C}^*(\iota(S))} \\
	& {B=\mathrm{C}^*(\phi(S))}
	\arrow["{\iota}", from=1-1, to=1-2, hook]
	\arrow["{\phi}"', from=1-1, to=2-2, hook]
	\arrow["{\sigma}", from=1-2, to=2-2, dashed, two heads]
\end{tikzcd}\]
\item We will say that a unital C*-cover $(A,\iota)$ of $S$ is {\em minimal} if for any unital C*-cover $(B,\phi)$ of $S$, there is a surjective unital *-homomorphism $\pi : B \to A$ such that $\pi \circ \phi = \iota$.
% https://q.uiver.app/?q=WzAsMyxbMCwwLCJTIl0sWzEsMCwiQT1cXG1hdGhybXtDfV4qKFxcaW90YShTKSkiXSxbMSwxLCJCPVxcbWF0aHJte0N9XiooXFxwaGkoUykpIl0sWzAsMSwiXFxpb3RhIiwwLHsic3R5bGUiOnsidGFpbCI6eyJuYW1lIjoiaG9vayIsInNpZGUiOiJ0b3AifX19XSxbMCwyLCJcXHBoaSIsMix7InN0eWxlIjp7InRhaWwiOnsibmFtZSI6Imhvb2siLCJzaWRlIjoidG9wIn19fV0sWzIsMSwiXFxzaWdtYSIsMix7InN0eWxlIjp7ImhlYWQiOnsibmFtZSI6ImVwaSJ9fX1dXQ==
\[\begin{tikzcd}
	{S} & {A=\mathrm{C}^*(\iota(S))} \\
	& {B=\mathrm{C}^*(\phi(S))}
	\arrow["{\iota}", from=1-1, to=1-2, hook]
	\arrow["{\phi}"', from=1-1, to=2-2, hook]
	\arrow["{\pi}"', from=2-2, to=1-2, dashed, two heads]
\end{tikzcd}\]
\end{enumerate}

The existence and uniqueness of the maximal unital C*-cover of a unital operator system was established by Kirchberg and Wassermann \cite{KW1998}. The following result is non-trivial. It is implied by \cite{DK2019}*{Theorem 4.4.3}.

\begin{theorem} \label{thm:maximal-unital-c-star-cover}
Let $K$ be a compact nc convex set. The maximal unital C*-cover for the unital operator system $\rmA(K)$ is the C*-algebra $\rmC(K)$ of continuous nc functions on $K$.
\end{theorem}

The existence and uniqueness of the minimal unital C*-cover of a unital operator system was established by Hamana \cite{Ham1979}. The results in \cite{DK2019} and \cite{KS2019} imply a description in terms of the nc state space of the operator system, which we will now describe.

Let $K$ be a compact nc convex set. It follows from Theorem \ref{thm:maximal-unital-c-star-cover} that there is a surjective *-homomorphism $\pi$ from $\rmC(K)$ onto the minimal unital C*-cover of $\rmA(K)$. A result of Dritschel and McCullough \cite{DM2005} implies that $\ker \pi$ is the {\em boundary ideal} in $\rmC(K)$ relative to $\rmA(K)$, i.e.\ the unique largest ideal in $\rmC(K)$ with the property that the restriction of the corresponding quotient *-homomorphism to $\rmA(K)$ is completely isometric.

Let $I_{\sb} = \ker \pi$ and let $\rmC(\sb) = \rmC(K)/I_{\sb}$. We will refer to $\rmC(\sb)$ as the minimal unital C*-cover of $\rmA(K)$. In order to explain this choice of notation and give a description of $\rmC(\sb)$ in terms of $K$, we require the spectral topology from \cite{KS2019}*{Section 9}.

\begin{definition}
Let $K$ be a compact nc convex set. We will say that a point $x \in K$ is {\em reducible} if $x$ is unitarily equivalent to a direct sum $x \simeq y \oplus z$ for points $y,z \in K$. We will say that $x$ is {\em irreducible} if it is not reducible, and we will write $\Irr(K)$ for the set of irreducible points in $K$. 
\end{definition}

\begin{remark}
Note that a point $x \in K$ is irreducible if and only if the corresponding *-homomorphism $\delta_x$ is. In particular, $\partial K \subseteq \operatorname{Irr}(K)$.
\end{remark}

Let $K$ be a compact nc convex set. Let $\Spec(\rmC(K))$ denote the C*-algebraic spectrum of $\rmC(K)$, i.e.\ the set of unitary equivalence classes of irreducible representations of $\rmC(K)$ equipped with the hull-kernel topology. For a point $x \in \operatorname{Irr}(K)$, we have already observed that the *-homomorphism $\delta_x$ is irreducible. Hence letting $[\delta_x]$ denote the unitary equivalence class of $\delta_x$, $[\delta_x] \in \Spec(\rmC(K))$. Note that the map $\operatorname{Irr}(K) \to \Spec(\rmC(K)) : x \to [\delta_x]$ is surjective.

\begin{definition}
The {\em spectral topology} on $\Irr(K)$ is the pullback of the hull-kernel topology on $\Spec(\rmC(K))$. Specifically, the open subsets of $\Irr(K)$ are the preimages of open subsets of $\Spec(\rmC(K))$ under the map $\Irr(K) \to \Spec(\rmC(K)) : x \to [\delta_x]$.
\end{definition}

The results in \cite{KS2019}*{Section 9} imply that
\[
I_{\sb} = \{f \in \rmC(K) : f(x) = 0 \text{ for all } x \in \overline{\partial K}\},
\]
where $\overline{\partial K}$ denotes the closure of $\partial K$ with respect to the spectral topology on $\operatorname{Irr}(K)$.

%%%
\subsection{Minimal and maximal C*-covers} \label{sec:minimal-maximal-non-unital-c-star-covers}
%%%

Connes and van Suijlekom \cite{CS2020} introduced an analogue for operator systems of a unital C*-cover of a unital operator system from Section~\ref{sec:min-max-c-star-covers}, which they refer to as a \csharp-cover. We will instead refer to C*-covers.
\begin{definition}
Let $S$ be an operator system.
\begin{enumerate}
\item We will say that a pair $(A,\iota)$ consisting of a C*-algebra $A$ and an embedding $\iota : S \to A$ is a {\em C*-cover} of $S$ if $A = \ca(\iota(S))$.
\item If $(A',\iota')$ is another C*-cover of $S$, then we will say that $(A,\iota)$ and $(A,\iota')$ are {\em equivalent} if there is a *-isomorphism $\pi : A \to A'$ such that $\pi \circ \iota = \iota'$.
\item We will say that a C*-cover $(A,\iota)$ of $S$ is {\em maximal} if for any C*-cover $(B,\phi)$ of $S$ there is a surjective *-homomorphism $\sigma : A \to B$ such that $\phi = \sigma \circ \iota$.
% https://q.uiver.app/?q=WzAsMyxbMCwwLCJTIl0sWzEsMCwiQT1cXG1hdGhybXtDfV4qKFxcaW90YShTKSkiXSxbMSwxLCJCPVxcbWF0aHJte0N9XiooXFxwaGkoUykpIl0sWzAsMSwiXFxpb3RhIiwwLHsic3R5bGUiOnsidGFpbCI6eyJuYW1lIjoiaG9vayIsInNpZGUiOiJ0b3AifX19XSxbMCwyLCJcXHBoaSIsMix7InN0eWxlIjp7InRhaWwiOnsibmFtZSI6Imhvb2siLCJzaWRlIjoidG9wIn19fV0sWzEsMiwiXFxzaWdtYSIsMCx7InN0eWxlIjp7ImJvZHkiOnsibmFtZSI6ImRhc2hlZCJ9LCJoZWFkIjp7Im5hbWUiOiJlcGkifX19XV0=
\[\begin{tikzcd}
	{S} & {A=\mathrm{C}^*(\iota(S))} \\
	& {B=\mathrm{C}^*(\phi(S))}
	\arrow["{\iota}", from=1-1, to=1-2, hook]
	\arrow["{\phi}"', from=1-1, to=2-2, hook]
	\arrow["{\sigma}", from=1-2, to=2-2, dashed, two heads]
\end{tikzcd}\]
\item We will say that a C*-cover $(A,\iota)$ of $S$ is {\em minimal} if for any C*-cover $(B,\phi)$ of $S$, there is a surjective *-homomorphism $\pi : B \to A$ such that $\pi \circ \phi = \iota$.
% https://q.uiver.app/?q=WzAsMyxbMCwwLCJTIl0sWzEsMCwiQT1cXG1hdGhybXtDfV4qKFxcaW90YShTKSkiXSxbMSwxLCJCPVxcbWF0aHJte0N9XiooXFxwaGkoUykpIl0sWzAsMSwiXFxpb3RhIiwwLHsic3R5bGUiOnsidGFpbCI6eyJuYW1lIjoiaG9vayIsInNpZGUiOiJ0b3AifX19XSxbMCwyLCJcXHBoaSIsMix7InN0eWxlIjp7InRhaWwiOnsibmFtZSI6Imhvb2siLCJzaWRlIjoidG9wIn19fV0sWzIsMSwiXFxzaWdtYSIsMix7InN0eWxlIjp7ImhlYWQiOnsibmFtZSI6ImVwaSJ9fX1dXQ==
\[\begin{tikzcd}
	{S} & {A=\mathrm{C}^*(\iota(S))} \\
	& {B=\mathrm{C}^*(\phi(S))}
	\arrow["{\iota}", from=1-1, to=1-2, hook]
	\arrow["{\phi}"', from=1-1, to=2-2, hook]
	\arrow["{\pi}"', from=2-2, to=1-2, two heads]
\end{tikzcd}\]
\end{enumerate}
\end{definition}

\begin{remark}
Let $(K,z)$ be a pointed compact nc convex set. If $(A,\iota)$ is a C*-cover for $\rmA(K,z)$, then since $\phi$ is an embedding, the unitization $\phi^\sharp : \rmA(K) \to A^\sharp$ is an embedding. Hence $(A^\sharp,\iota^\sharp)$ is a unital C*-cover of $\rmA(K)$. 
\end{remark}

The existence and uniqueness of the minimal C*-cover of an operator system was established in \cite{CS2020}*{Theorem 2.2.5} under the name \csharp-envelope. In this section we will prove the existence and uniqueness of the maximal C*-cover, and we will describe the maximal and minimal C*-covers of an operator system in terms of the maximal and minimal unital C*-covers of its unitization.

\begin{proposition}
Let $S$ be an operator system. If the maximal and minimal C*-covers of $S$ exist, then they are unique up to equivalence.
\end{proposition}

\begin{proof}
Let $(A,\iota)$ and $(A',\iota')$ be maximal C*-covers for $S$. Then by definition there are surjective homomorphisms $\sigma : A \to A'$ and $\sigma' : A' \to A$ such that $\iota' = \sigma \circ \iota$ and $\iota = \sigma' \circ \iota'$. Hence $\sigma^{-1} = \sigma'$, so $\sigma$ is a *-isomorphism and hence $(A,\iota)$ and $(A,\iota')$ are equivalent. The proof for the minimal C*-cover is similar.
\end{proof}

\begin{theorem} \label{thm:min-max-c-star-algebras}
Let $(K,z)$ be a compact pointed nc convex set.
\begin{enumerate}
\item The C*-algebra $\rmC(K,z)$ is a maximal C*-cover for $\rmA(K,z)$ with respect to the canonical inclusion.
\item Let $I_{\sb}$ denote the boundary ideal in the C*-algebra $\rmC(K)$ of continuous nc functions on $K$ relative to $\rmA(K)$, so that the C*-algebra $\rmC(K)/I_{\sb} \cong \rmC(\sb)$ is the minimal unital C*-cover of $\rmA(K)$, and let $I_{(\sb,z)} = I_{\sb} \cap \rmC(K,z)$. Then the C*-algebra $\rmC(K,z)/I_{(\sb,z)}$ is the minimal C*-cover of $\rmA(K,z)$ with respect to the quotient *-homomorphism. In particular, the C*-algebra generated by the image of $\rmA(K,z)$ under the canonical embedding of $\rmA(K)$ into $\rmC(\sb)$ is isomorphic to $\rmC(\sb,z)$
\end{enumerate}
\end{theorem}

\begin{proof}
(1) Let $(B,\phi)$ be a C*-cover for $\rmA(K,z)$. We can assume that $B \subseteq \calM_n$ for some $n$, so that $\phi = x$ and $B = \delta_x(\rmC(K,z))$ for some $x \in K_n$. It follows that $\rmC(K,z)$ is a maximal C*-cover for $\rmA(K,z)$ with respect to the canonical inclusion. 

(2) Since $\rmC(K)/I_{\sb}$ is a unital C*-cover for $\rmA(K)$, $\rmC(K,z) / I_{(\sb,z)}$ is a C*-cover for $\rmA(K,z)$. To see that it is minimal, it suffices to show that if $(B,\phi)$ is any C*-cover for $\rmA(K,z)$, then $\ker \sigma \subseteq I_{(\sb,z)}$. 

By (1), there is a surjective unital *-homomorphism $\sigma : \rmC(K,z) \to B$ such that $\sigma|_{\rmA(K,z)} = \phi$. The unitization $\sigma^\sharp : \rmC(K) \to B^\sharp$ is a unital *-homomorphism satisfying $\sigma^\sharp|_{\rmA(K)} = \phi^\sharp$. Since $\phi$ is an embedding, $\phi^\sharp$ is completely isometric, so $\ker \sigma^\sharp \subseteq I_{\sb}$. Hence $\ker \sigma \subseteq I_{(\sb,z)}$.
\end{proof}

\begin{definition}
Let $(K,z)$ be a pointed compact nc convex set. Let $I_{(\sb,z)}$ denote the ideal in $\rmC(K,z)$ from Theorem \ref{thm:min-max-c-star-algebras} and let $\rmC(\sb,z) = \rmC(K,z)/I_{(\sb,z)}$. We will refer to $\rmC(\sb,z)$ as the {\em minimal C*-cover} of $\rmA(K,z)$, and we will refer to the corresponding quotient *-homomorphism as the {\em canonical embedding} of $\rmA(K,z)$ into $\rmC(\sb,z)$.
\end{definition}

\begin{remark} \label{rem:generalized-boundary-ideal}
The ideal $I_{(\sb,z)} = \ker \pi$ is a pointed analogue of the boundary ideal from Section \ref{sec:minimal-maximal-unital-c-star-covers}. It is the largest ideal in $\rmC(K,z)$ such that the corresponding quotient *-homomorphism restricts to an embedding of $\rmA(K,z)$.
\end{remark}

\begin{example}
Define $a,b \in \calM_2$ by
\[
a = \left[\begin{matrix} 1 & 0 \\ 0 & -1 \end{matrix}\right],\quad 
b = \left[\begin{matrix} 1 & 0 \\ 0 & -1/2 \end{matrix}\right].
\]
Let $S = \operatorname{span}\{a\}$ and $T = \operatorname{span}\{b\}$. Then $S$ and $T$ are nonunital operator systems, and it is not difficult to verify that $S$ and $T$ are isomorphic to the operator systems considered in Example \ref{ex:example-pointed-1} and Example \ref{ex:first-example-pointed-2} respectively.

Let $(K,z)$ denote the nc quasistate space of $S$. Note that this is the same $(K,z)$ from Example \ref{ex:example-pointed-1}. Since $K_1 = [-1,1]$ is a simplex, the results in \cite{KS2019} imply that $\partial K = \partial K_1 = \{-1,1\}$. Hence identifying $S$ with $\rmA(K,z)$, the minimal C*-cover of $S^\sharp = \rmA(K)$ is $\rmC(\sb) = \rmC(\{-1,1\}) \cong \C^2$. 

Let $\iota : \rmA(K) \to \rmC(\sb)$ denote the canonical embedding. Then $\iota(\rmA(K,z)) \cong \{(-\alpha,\alpha) : \alpha \in \C\} \cong \C$. Hence $\rmC(\sb,z) \cong \C$. Note that $\rmC(\sb,z)$ is unital even though $\rmA(K,z)$ is nonunital.

Define $\theta : S \to T$ defined by $\theta(\alpha a) = \alpha b$ for $\alpha \in \C$. Then arguing as in Example \ref{ex:first-example-pointed-2}, $\theta$ is an isomorphism. Hence the minimal C*-cover of $T$ is also isomorphic to $\C^2$.
\end{example}

\begin{example}
Let $A$ be a C*-algebra with nc quasistate space $(K,z)$. Then $A$ is clearly a C*-cover of itself with respect to the identity map. By definition, there is a surjective *-homomorphism $\pi : A \to \rmC(\sb,z)$ that is completely isometric on $A$. Therefore, $\pi$ is a *-isomorphism, implying $A = \rmC(\sb,z)$.
\end{example}

Let $K$ be a compact nc convex set. A useful fact implied by \cite{DK2015}*{Theorem 3.4} and \cite{DK2019}*{Proposition 5.2.4} is that the direct sum of the points in $\partial K$ extends to a faithful representation of the minimal unital C*-cover $\rmC(\sb)$. Specifically, define $y \in K$ by $y = \oplus_{x \in \partial k} x$. Then the *-homomorphism $\delta_y$ satisfies $\ker \delta_y = \ker I_{\sb}$. Hence $\delta_y(\rmC(K)) \cong \rmC(\sb)$. The following result is an analogue of this fact for the minimal C*-cover of an operator system.

\begin{proposition} \label{prop:factors-through}
Let $(K,z)$ be a pointed compact nc convex set. Define $y \in K$ by $y = \oplus_{x \in \partial K \setminus \{z\}} x$. Then the *-homomorphism $\delta_y$ satisfies $\ker \delta_y = I_{(\sb,z)}$, where $I_{(\sb,z)}$ is the ideal from Theorem \ref{thm:min-max-c-star-algebras}. Hence $\delta_y(\rmC(K,z)) \cong \rmC(\sb,z)$.
\end{proposition}

\begin{proof}
From above, $(\delta_y \oplus \delta_z)(\rmC(K)) \cong \rmC(\sb)$. So considered as a *\nobreakdash-representation of $\rmC(K)$, $\ker (\delta_y \oplus \delta_z) = I_{\sb}$. By Theorem \ref{thm:min-max-c-star-algebras}, $\ker (\delta_y \oplus \delta_z) \cap \rmC(K,z) = I_{(\sb,z)}$. Since $\delta_y$ is zero on $\rmA(k,z)$ and so also on $\rmC(K,z)$, it follows that $\ker \delta_y \cap \rmC(K,z) = I_{(\sb,z)}$. Hence by Theorem \ref{thm:min-max-c-star-algebras}, $\delta_y(\rmC(K,z)) \cong \rmC(\sb,z)$.
\end{proof}

We will say more about the minimal C*-cover in Section~\ref{sec:characterization-unital-operator-systems}.

%%%
\section{Characterization of unital operator systems} \label{sec:characterization-unital-operator-systems}
%%%

In this section we will apply the results from Section \ref{sec:min-max-c-star-covers} to establish a characterization of operator systems that are unital in terms of their nc quasistate space. We note that a closely related problem, of characterizing operator spaces that are unital operator systems, has been considered by Blecher and Neal \cite{BN2011}. 

\begin{theorem} \label{thm:characterization-unit}
Let $(K,z)$ be a pointed compact nc convex set. The following are equivalent for a pointed continuous affine nc function $e \in \rmA(K,z)$:
\begin{enumerate}
\item The function $e$ is a distinguished archimedean matrix order unit for $\rmA(K,z)$.
\item The image of $e$ under the canonical embedding of $\rmA(K,z)$ into $\rmC(\sb,z)$ is the identity.
\item For every $n$ and every $x \in (\partial K \setminus \{z\})_n$, $e(x) = 1_n$.
\end{enumerate}
\end{theorem}

\begin{proof}
(1) $\Rightarrow$ (2)
Suppose that $e$ is a distinguished archimedean matrix order unit for $\rmA(K,z)$. Then $\rmA(K,z)$ is a unital operator system, so it follows from \cite{CE1977}*{Theorem 4.4} that there is $y \in K_n$ such that $y$ is a unital complete isometry on $\rmA(K,z)$ with $e(y) = 1_n$. Then $\ker \delta_y$ is contained in the boundary ideal $I_{\sb,z}$ from Remark \ref{rem:generalized-boundary-ideal}. It follows that the canonical embedding of $\rmA(K,z)$ into $\rmC(\sb,z)$ factors through $y$, and hence maps $e$ to the identity.

(2) $\Rightarrow$ (3)
Suppose that the image of $e$ under the canonical embedding of $\rmA(K,z)$ into $\rmC(\sb,z)$ is the identity. Proposition \ref{prop:factors-through} implies that the restriction to $\rmA(K,z)$ of every nc quasistate in $\partial K \setminus \{z\}$ factors through $\rmC(\sb,z)$. It follows that for $x \in (\partial K \setminus \{z\})_n$, $e(x) = 1$. 

(3) $\Rightarrow$ (1)
Suppose that for every $n$ and every $x \in (\partial K \setminus \{z\})_n$, $e(x) = 1_n$. Then it follows from Proposition \ref{prop:factors-through} that the image of $e$ under the canonical embedding of $\rmA(K,z)$ into $\rmC(\sb,z)$ is the identity. It follows that $e$ is a distinguished archimedean matrix order unit for $\rmC(\sb,z)$, and hence also for $\rmA(K,z)$.
\end{proof}

\begin{corollary} \label{cor:characterization-unital}
Let $S$ be an operator system with nc quasistate space $(K,z)$. The following are equivalent:
\begin{enumerate}
\item The operator system $S$ is unital. 
\item There is $e \in S$ such that for every $n$ and every $x \in (\partial K \setminus \{z\})_n$, $e(x) = 1_n$
\end{enumerate}
\end{corollary}

The next result is \cite{CS2020}*{Theorem 2.25 (ii)}. 

\begin{corollary}
Let $S$ be a unital operator system. Then the minimal unital C*-cover of $S$ and the minimal C*-cover of $S$ coincide.
\end{corollary}

\begin{proof}
Let $(A,\iota)$ and $(B,\kappa)$ denote the minimal unital C*-cover of $S$ and the minimal C*-cover of $S$ respectively. It follows from Theorem \ref{thm:characterization-unit} and Proposition \ref{prop:factors-through} that $B$ is unital. Hence by the universal property of $A$, there is a surjective *-homomorphism $\pi : A \to B$ such that $\pi \circ \iota = \kappa$. On the other hand, by the universal property of $B$, there is a surjective *-homomorphism $\sigma : B \to A$ such that $\sigma \circ \kappa = \iota$. Hence $A$ and $B$ are isomorphic. 
\end{proof}

%%%
\section{Quotients of operator systems} \label{sec:quotients-operator-systems}
%%%

In this section we will utilize the dual equivalence between the category of operator systems and the category of pointed compact nc convex sets to develop a theory of quotients for operator systems. We will show that the theory developed here extends the theory of quotients for unital operator systems developed by Kavruk, Paulsen, Todorov and Tomforde \cite{KPTT2013}. We note that the theory of quotients for unital operator systems can be developed in a similar way using the dual equivalence between the category of unital operator systems and the category of compact nc convex sets from \cite{DK2019}*{Section 3}.

\begin{definition} \label{defn:kernel}
Let $S$ be an operator system and let $(K,z)$ denote the nc quasistate space of $S$. We will say that a subset $J \subseteq S$ is a {\em kernel} if there is an nc quasistate $x \in K$ such that $J = \ker x$.
\end{definition}

\begin{remark}
For $x \in K_n$, the closure of the image $x(S) \subseteq \calM_n$ is an operator system. Hence $J$ is a kernel if and only if there is an operator system $T$ and a completely contractive and completely positive map $\phi : S \to T$ with $\ker \phi = J$. 
\end{remark}

Let $S$ be an operator system and let $(K,z)$ denote the nc quasistate space of $S$. For a subset $Q \subseteq S$, the {\em annihilator} of $Q$ is $Q^\perp = \{x \in K : a(x) = 0 \text{ for all } a \in Q \}$. Note that $Q^\perp$ is a closed nc convex set. Similarly, for a subset $X \subseteq K$, the {\em annihilator} of $X$ is $X^\perp = \{a \in S : a(x) = 0 \text{ for all } x \in X \}$. 

The next result is a noncommutative analogue of \cite{Alf1971}*{II.5.3}.

\begin{lemma}
Let $K$ be a compact nc convex set and let $X \subseteq K$ be a subset. Then $X^{\perp \perp} = Y \cap K$, where $Y \subseteq \calM(\rmA(K)^{*})$ denotes the closed nc convex hull generated by $\sqcup \operatorname{span} X_n$, where $\operatorname{span} X_n$ is taken in $\calM_n(\rmA(K)^{*})$.
\end{lemma}

\begin{proof}
It is clear that $Y \cap K \subseteq X^{\perp \perp}$. For the other inclusion, suppose for the sake of contradiction there is $z \in (X^{\perp \perp})_n \setminus (Y \cap K)$. Then $z \notin Y$. Hence by the nc separation theorem \cite{DK2019}*{2.4.1}, there is a self-adjoint element $a \in \rmA(K)$ satisfying $a(z) \not \leq 1_n \otimes 1_n$ but $a(y) \leq 1_n \otimes 1_p$ for all $y \in Y_p$. Since each $Y_p$ is a subspace, this forces $a(y) = 0$ for all $y \in Y_p$. Hence viewing $a$ as an $n \times n$ matrix $a = (a_{ij})$ over $\rmA(K)$, $a_{ij}(y) = 0$ for all $y \in Y$. In particular, $a_{ij}(x) = 0$ for all $x \in X$. Hence $a_{ij} \in X^{\perp}$ for all $i,j$. Since $z \in X^{\perp \perp}$, it follows that $a_{ij}(z) = 0$ for all $i,j$. Therefore, $a(z) = 0$, giving a contradiction.
\end{proof}

\begin{proposition} \label{prop:kernel-restriction-map}
Let $(K,z)$ be a pointed compact nc convex set. A subset $J \subseteq \rmA(K,z)$ is a kernel if and only if $J = J^{\perp \perp}$. If $J$ is a kernel and $M=J^\perp$, then the completely contractive completely positive restriction map $\rmA(K,z) \to \rmA(M,z)$ has kernel $J$. Moreover, $z \in M$ and the pair $(M,z)$ is a pointed compact nc convex set.
\end{proposition}

\begin{proof}
Suppose that $J = J^{\perp \perp}$. Let $M = J^\perp$. Then $J = M^\perp$. Let $r : \rmA(K,z) \to \rmA(M,z)$ denote the restriction map. Then $r$ is completely contractive and completely positive and $\ker r = M^\perp = J$. Hence $J$ is a kernel.

Conversely, suppose that $J$ is a kernel. It is clear that $J \subseteq J^{\perp \perp}$. For the other inclusion, choose $x \in K$ such that $J = \ker x$. Let $T$ denote the closure of the image $\rmA(K,z)(x) \subseteq \calM_n$. Then $T$ is an operator system. Letting $(L,w)$ denote the nc quasistate space of $T$, we can identify $T$ with $\rmA(L,w)$. Let $\psi : L \to K$ denote the continuous affine map obtained by applying Theorem \ref{thm:categorical-equivalence} to $x$. Then for $a \in J$ and $y \in L$, $0 = a(x)(y) = a(\psi(y))$. Hence $\psi(L) \subseteq J^\perp$, so for $a \in J^{\perp \perp}$ and $y \in L$, $0 = a(\psi(y)) = a(x)(y)$, i.e.\ $a(x) = 0$. Hence $J^{\perp \perp} \subseteq J$, so $J = J^{\perp \perp}$. 

If $J$ is a kernel and $M = J^\perp$, then clearly $z \in M$. To see that $(M,z)$ is a pointed compact nc convex set, let $\theta : \rmA(M,z) \to \calM_n$ be an nc quasistate. Let $r : \rmA(K,z) \to \rmA(M,z)$ denote the restriction map from above. Then the composition $\theta \circ r$ is an nc quasistate on $\rmA(K,z)$. Since $(K,z)$ is a pointed compact nc convex set, by definition there is $x \in K$ such that $\theta \circ r = x$. Since $x$ factors through $r$, $x \in J^\perp = M$. 
\end{proof}

\begin{definition} \label{defn:quotient}
Let $S$ be an operator system and let $(K,z)$ denote the nc quasistate space of $S$. For a kernel $J \subseteq S$, we let $S/J$ denote the operator system $\rmA(M,z)$, where $M = J^\perp$. We will refer to $S/J$ as the {\em quotient} of $S$ by $J$, and we will refer to the restriction map $S \to \rmA(M,z)$ obtained by identifying $S$ with $\rmA(K,z)$ as the {\em canonical quotient map}.
\end{definition}

\begin{remark}
Note that we have applied Theorem \ref{thm:duality-theorem} to identify $S$ with $\rmA(K,z)$. It is clear that the canonical quotient map $\rmA(K,z) \to \rmA(M,z)$ is completely contractive and completely positive.
\end{remark}

The next result characterizes operator system quotients in terms of a natural universal property. It is an analogue of \cite{KPTT2013}*{Proposition 3.6}. 

\begin{theorem} \label{thm:universal-characterization-quotients}
Let $S$ be an operator system and let $J \subseteq S$ be a kernel. The quotient $S/J$ is the unique operator system up to isomorphism satisfying the following universal property: there is a completely contractive and completely positive map $\phi : S \to S/J$, and whenever $T$ is an operator system and $\psi : S \to T$ is a completely contractive and completely positive map with $J \subseteq \ker \psi$, then $\psi$ factors through $\phi$. In other words, there is a completely contractive and completely positive map $\omega : S/J \to T$ such that $\psi = \omega \circ \phi$.
% https://q.uiver.app/?q=WzAsNSxbMSwwLCJTIl0sWzIsMCwiUy9KIl0sWzIsMSwiVCJdLFswLDAsIkoiXSxbMywxXSxbMywwLCIiLDAseyJzdHlsZSI6eyJ0YWlsIjp7Im5hbWUiOiJob29rIiwic2lkZSI6InRvcCJ9fX1dLFswLDEsIlxccGhpIl0sWzEsMiwiXFxvbWVnYSIsMCx7InN0eWxlIjp7ImhlYWQiOnsibmFtZSI6ImVwaSJ9fX1dLFswLDIsIlxccHNpIiwyXV0=
\[\begin{tikzcd}
	{J} & {S} & {S/J} \\
	&& {T} & {}
	\arrow[from=1-1, to=1-2, hook]
	\arrow["{\phi}", from=1-2, to=1-3]
	\arrow["{\omega}", from=1-3, to=2-3]
	\arrow["{\psi}"', from=1-2, to=2-3]
\end{tikzcd}\]
\end{theorem}

\begin{proof}
To see that $S/J$ satisfies this universal property, first note that the canonical quotient map $\phi : S \to S/J$ is completely contractive and completely positive. Let $T$ be an operator system and let $\psi : S \to T$ be a completely contractive and completely positive map with $J \subseteq \ker \psi$. Letting $(K,z)$ and $(L,w)$ denote the nc quasistate spaces of $S$ and $T$ respectively, we can assume that $S = \rmA(K,z)$ and $T = \rmA(L,w)$. Let $\psi^d : L \to K$ denote the continuous affine nc map obtained by applying Theorem \ref{thm:duality-theorem} to $\psi$.

Let $M = J^\perp$. For $a \in J$ and $y \in L$, the fact that $J \subseteq \ker \psi$ implies that $0 = \psi(a)(y) = a(\psi^d(y))$. Hence $\psi^d(L) \subseteq J^\perp = M$. Restricting the codomain of $\psi^d$ to $M$ and applying Theorem \ref{thm:duality-theorem} to $\psi^d$, we obtain a completely contractive and completely positive map $\omega : \rmA(M,z) \to \rmA(L,w)$ such that $\omega \circ \phi = \psi$.

To see that $S/J$ is the unique operator system with this universal property, suppose that $R$ is another operator system that satisfies the property from the statement of the theorem, then there are surjective completely contractive and completely positive maps $S/J \to R$ and $R \to S/J$ such that the composition is the identity map on $S/J$. It follows that each of the individual maps must be a completely isometric complete order isomorphism. Hence $R$ is isomorphic to $S/J$.
\end{proof}

In order to relate our theory of quotients of operator systems to the theory of quotients of unital operator systems from \cite{KPTT2013}, we require the following result.

\begin{lemma} \label{cor:unital-op-sys-unit}
Let $(K,z)$ be a pointed compact nc convex set such that $\rmA(K,z)$ is a unital operator system and let $e \in \rmA(K,z)$ denote the distinguished archimedean matrix order unit. Let $J \subseteq \rmA(K,z)$ be a kernel and let $M = J^\perp$. Then for $x \in \partial M \setminus \{z\}$, $e(x) = 1$.
\end{lemma}

\begin{proof}
Let $K^0$ and $K^1$ denote the closed nc convex hulls of $\{z\}$ and $\partial K \setminus \{z\}$ respectively. By Theorem \ref{thm:characterization-unit}, $e(x) = 1$ for $x \in \partial K \setminus \{z\}$. Hence by the continuity of $e$, $e(x) = 1$ for all $x \in K^1$. Since $e(z) = 0$, in particular this implies that $K^0 \cap K^1 = \emptyset$. 

%$K^0 \not \subseteq K^1$. Since every point in $K^0$ is a direct sum of copies of $z$, it follows that $z \in \partial K^0$.

For $x \in \partial K$, either $e(x) = 1$ or $e(x) = 0$. It follows from \cite{DK2015}*{Theorem 3.4} and \cite{DK2019}*{Proposition 5.2.4} that the image of $e$ under the canonical embedding of $\rmA(K)$ into its minimal unital C*-cover $\rmC(\sb)$ is a projection in the center of $\rmC(\sb)$.

Choose $x \in K_m$ and let $y \in K_n$ be a maximal dilation of $x$. Then there is an isometry $\alpha \in \calM_{n,m}$ such that $x = \alpha^* y \alpha$. By \cite{DK2019}*{Proposition 5.2.4}, the *-homomorphism $\delta_y$ factors through $\rmC(\sb)$. Hence from above, $y$ decomposes as a direct sum $y = y_0 \oplus y_1$ for $y_0 \in K^0_{n_0}$ and $y_1 \in K^1_{n_1}$. This implies that $x$ can be written as an nc convex combination $x = \alpha_0^* y_0 \alpha_0 + \alpha_1^* y_1 \alpha_1$ for $\alpha_0 \in \calM_{n_0,m}$ and $\alpha_1 \in \calM_{n_1,m}$ satisfying $\alpha_0^* \alpha_0 + \alpha_1^* \alpha_1 = 1_m$.

Suppose $x \in (\partial M)_m$ and write $x$ can be written as an nc convex combination $x = \alpha_0^* y_0 \alpha_0 + \alpha_1^* y_1 \alpha_1$ as above. Then for $a \in J$,
\[
0 = a(x) = \alpha_0^* a(y_0) \alpha_0 + \alpha_1^* a(y_1) \alpha_1.
\]
Since $y_0$ is a direct sum of copies of $z$, $a(y_0) = 0$. Hence $\alpha_1^* a(y_1) \alpha_1 = 0$.

From above, we can decompose $y_1$ with respect to the range of $\alpha$ as
\[
y_1 = \left[\begin{matrix} u_1 & \ast \\ \ast & \ast \end{matrix}\right]
\]
for $u_1 \in M_k$, and there is $\beta_1 \in \calM_{k,m}$ such that $\beta_1^* u_1 \beta_1 = \alpha_1^* y_1 \alpha_1$ and $\alpha_0^* \alpha_0 + \beta_1^* \beta_1 = 1_m$. 

Now since $y_0,u_1 \in M$ and $x \in (\partial M)_m$, it follows that either $\alpha_0 = 0$ or $\beta_1 = 0$. Hence either $x \in K^0$ or $x \in K^1$. In the former case, $x = z$, while in the latter case, $e(x) = 1$.
\end{proof}

\begin{proposition} \label{prop:quotient-unital-op-sys}
Let $S$ be a unital operator system. Then for every kernel $J \subseteq S$, the quotient operator system $S/J$ is unital.
\end{proposition}

\begin{proof}
Letting $(K,z)$ denote the nc quasistate space of $S$, we can assume that $S = \rmA(K,z)$. Let $M = J^\perp$, so that $S/J = \rmA(M,z)$, and let $\phi : \rmA(K,z) \to \rmA(M,z)$ denote the canonical quotient map. Let $e \in \rmA(K,z)$ denote the distinguished archimedean matrix order unit. Then for $x \in \partial M \setminus \{z\}$, Corollary \ref{cor:unital-op-sys-unit} implies that $e(x) = 1$. Hence by Theorem \ref{thm:characterization-unit}, $\phi(e)$ is an archimedean matrix order unit.
\end{proof}

\begin{remark}
If $S$ is a unital operator system and $J$ is the kernel of a unital completely positive map, then the quotient $S/J$ from Definition \ref{defn:quotient} coincides with the definition of quotient in \cite{KPTT2013}. Indeed, the quotient $T$ of $S$ by $J$ that they consider in their paper is the unique unital operator system satisfying a universal property analogous to the property in Theorem \ref{thm:universal-characterization-quotients} for unital completely positive maps into unital operator systems. By Proposition \ref{prop:quotient-unital-op-sys}, $S/J$ is a unital operator system, it follows from Theorem \ref{thm:universal-characterization-quotients} that $S/J = T$.
\end{remark}

\begin{lemma} \label{lem:ideal-generated-by-kernel}
Let $(K,z)$ be a pointed compact nc convex set and let $J \subseteq \rmA(K,z)$ be a kernel. Let $M = J^\perp$. Then the closed two-sided ideal $I$ of $\rmC(K,z)$ generated by $J$ is $I = \{ f \in \rmC(K,z) : f|_M = 0 \}$. Hence letting $\pi : \rmC(K,z) \to \rmC(M,z)$ denote the restriction *-homomorphism, $I = \ker \pi$. 
\end{lemma}

\begin{proof}
Let $I' = \ker \pi$. Then $I' = \{ f \in \rmC(K,z) : f|_M = 0 \}$. By Proposition \ref{prop:kernel-restriction-map}, the restriction $\pi|_{\rmA(K,z)}$ satisfies $\ker \pi|_{\rmA(K,z)} = J$, so it is clear that $I \subseteq I'$.

For the other inclusion, first note that $J = I' \cap \rmA(K,z)$. Hence by the definition of $I$ and the fact from above that $I \subseteq I'$,
\[
J \subseteq I \cap \rmA(K,z) \subseteq I' \cap \rmA(K,z) = J,
\]
implying $J = I \cap \rmA(K,z)$. Let $\rho : \rmC(K,z) \to \rmC(K,z)/I$ denote the quotient *-homomorphism. Since the restriction $\rho|_{\rmA(K,z)}$ has kernel $J$, Theorem \ref{thm:universal-characterization-quotients} implies that $\rho|_{\rmA(K,z)}$ factors through $\rmA(M,z)$. It follows that there is a completely contractive and completely positive map $\omega : \rmA(M,z) \to \rmC(K,z)/I$ such that $\omega \circ \pi|_{\rmA(K,z)} = \rho|_{\rmA(K,z)}$.

By the universal property of $\rmC(M,z)$, $\omega$ extends to a *-homomorphism $\sigma : \rmC(M,z) \to \rmC(K,z)/I$. Hence $I' \subseteq I$, and we conclude that $I' = I$.  
\end{proof}

\begin{proposition} \label{prop:kernel-max-c-star-cover}
Let $(K,z)$ be a pointed compact nc convex set. Let $J \subseteq \rmA(K,z)$ be a subset and let $I$ denote the closed two-sided ideal of $\rmC(K,z)$ generated by $J$. Then $J$ is a kernel if and only if $I \cap \rmA(K,z) = J$.
\end{proposition}

\begin{proof}
If $J$ is a kernel, then letting $M = J^\perp$, Proposition \ref{prop:kernel-restriction-map} and Lemma \ref{lem:ideal-generated-by-kernel} imply that $I \cap \rmA(K,z) = \{ a \in \rmA(K,z) : a|_M = 0 \} = M^\perp = J$. Conversely, if $I \cap \rmA(K,z) = J$, then letting $\pi : \rmC(K,z) \to \rmC(K,z)/I$ denote the quotient *-homomorphism, $J = \ker \pi|_{\rmA(K,z)}$. Hence $J$ is a kernel. 
\end{proof}

%%%
\section{C*-simplicity} \label{sec:c-star-simplicity}
%%%

In this section we will establish a characterization of operator systems with the property that their minimal C*-cover (i.e.\ their C*-envelope) is simple. The characterization will be in terms of the nc quasistate space of an operator system.

\begin{definition}
We will say that an operator system $S$ is {\em C*-simple} if its minimal C*-cover $\camin(S)$ is simple. 
\end{definition}

We will require the {\em spectral topology} on the irreducible points in a compact nc convex set from Section \ref{sec:minimal-maximal-unital-c-star-covers}), which was introduced in \cite{KS2019}*{Section 9}. Recall that for a compact nc convex set $K$, the spectral toplogy on the set $\Irr(K)$ of irreducible points in $K$ is defined in terms of the hull-kernel topology on the spectrum of the C*-algebra $\rmC(K)$. 

By Proposition \ref{prop:factors-through}, letting $y = \oplus_{x \in \partial K \setminus \{z\}} x$, the kernel of the *-homomorphism $\delta_y$ on $\rmC(K,z)$ is the boundary ideal $I_{(\sb,z)}$ from Theorem \ref{thm:min-max-c-star-algebras}. In particular, the quotient $\rmC(K,z) / I_{(\sb,z)}$ is isomorphic to the minimal C*-cover $\rmC(\sb,z)$ of $\rmA(K,z)$. The proof of the following result now follows exactly as in the proof of \cite{KS2019}*{Proposition 9.4}.

\begin{proposition} \label{prop:spectral-closure-factors-through}
Let $(K,z)$ be a pointed compact nc convex set. A point $x \in \Irr(K)$ belongs to the closure of $\partial K \setminus \{z\}$ with respect to the spectral topology if and only if the corresponding representation $\delta_x : \rmC(K,z) \to \calM_n$ factors through the minimal C*-cover $\rmC(\sb,z)$ of $\rmA(K,z)$.
\end{proposition}

\begin{theorem} \label{thm:c-star-simple}
Let $(K,z)$ be a pointed compact nc convex set. The operator system $\rmA(K,z)$ is C*-simple if and only if the closed nc convex hull of any nonzero point in the spectral closure of $\partial K$ contains $\partial K \setminus \{z\}$.
\end{theorem}

\begin{proof}
Suppose that $\rmA(K,z)$ is C*-simple, so that its minimal C*-cover $\rmC(\sb,z)$ is simple. Choose nonzero $x \in K_m$ in the spectral closure of $\partial K \setminus \{z\}$ and let $M \subseteq K$ denote the closed nc convex hull of $x$. Suppose for the sake of contradiction there is $y \in (\partial K)_n \setminus \{z\}$ such that $y \notin M$.

By Proposition \ref{prop:spectral-closure-factors-through}, the corresponding representation $\delta_x : \rmC(K,z) \to \calM_n$ factors through the minimal C*-cover $\rmC(\sb,z)$ of $\rmA(K,z)$. Since $\rmC(\sb,z)$ is simple, it follows that the kernel of $\delta_x$ is the boundary ideal $I_{(\sb,z)}$ from Theorem \ref{thm:min-max-c-star-algebras}, so the range of $\delta_x$ is isomorphic to the minimal C*-cover $\rmC(K,z)/I_{(\sb,z)} \cong \rmC(\sb,z)$. In particular, $x$ is an embedding. Similarly, $y$ is an embedding.

By the nc separation theorem \cite{DK2019}*{Corollary 2.4.2}, there is self-adjoint $a \in \calM_n(\rmA(K,z))$ and self-adjoint $\gamma \in \calM_n$ such that $a(y) \not \leq \gamma \otimes 1_n$ but $a(u) \leq \gamma \otimes 1_p$ for $u \in M_p$. In particular, $a(x) \leq \gamma \otimes 1_m$ but $a(y) \not \leq \gamma \otimes 1_n$. However, from above $x$ and $y$ are embeddings, meaning that they are complete order embeddings on $\rmA(K)$, giving a contradiction.

Conversely, suppose that the closed nc convex hull of any nonzero point in the spectral closure of $\partial K$ contains $\partial K \setminus \{z\}$. Let $I$ be a proper ideal in $\rmC(\sb,z)$ and choose nonzero irreducible $y \in K_n$ such that the *-homomorphism $\delta_y$ on $\rmC(K,z)$ factors through $\rmC(\sb,z)/I$. Then by Proposition \ref{prop:spectral-closure-factors-through}, $y$ is in the spectral closure of $\partial K$. Hence by assumption the closed nc convex hull of $y$ contains $\partial K \setminus \{z\}$. 

By \cite{DK2019}*{Theorem 6.4.3}, every point in $\partial K$ is a limit of compressions of $y$. Hence, replacing $y$ with a sufficiently large amplification, there are isometries $\alpha_i \in \calM_{p,n}$ such that $\lim \alpha_i^* y \alpha_i = \oplus_{x \in \partial K \setminus \{z\}} x$. By passing to a subnet we can assume that there is an nc state $\mu$ on $\rmC(K)$ such that the *-homomorphism $\delta_y$ satisfies $\lim \alpha_i^* \delta_y \alpha_i = \mu$ in the nc state space of $\rmC(K)$. Then since $\mu|_{\rmA(K)} = \oplus_{x \in \partial K \setminus \{z\}} x$, and since extreme points in $K$ have unique extensions to nc states on $\rmC(K)$, $\mu$ is the *-homomorphism $\mu = \oplus_{x \in \partial K \setminus \{z\}} \delta_x$ (see \cite{DK2019}*{Theorem 6.1.9}).

By Proposition \ref{prop:factors-through}, the image of $\rmC(K,z)$ under this *-homomorphism is isomorphic to $\rmC(\sb,z)$. It follows that the canonical *-homomorphism from $\rmC(K,z)$ onto $\rmC(\sb,z)$ factors through $\delta_y$. Hence $I = 0$. Since $I$ was arbitrary, we conclude that $\rmC(\sb,z)$ is simple.
\end{proof}

The following corollary applies when the set $\partial K$ of extreme points of $K$ is closed in the spectral topology. This is equivalent to the statement that every nonzero irreducible representation of $\rmC(\sb,z)$ restricts to an extreme point of $K$.

\begin{corollary}
Let $(K,z)$ be a pointed nc convex set such that $\partial K$ is closed in the spectral topology. Then $\rmA(K,z)$ is C*-simple if and only if for every nonzero compact nc convex subset $M \subseteq K$, either $M \cap \partial K = \emptyset$ or $M \cap \partial K = \partial K$.
\end{corollary}

\begin{proof}
Suppose that $\rmA(K,z)$ is C*-simple. If $M \cap \partial K \ne \emptyset$ then Theorem \ref{thm:c-star-simple} implies that $\partial K \subseteq M$. Conversely, suppose that for every nonzero compact nc convex subset $M \subseteq K$, either $M \cap \partial K = \emptyset$ or $M \cap \partial K = \partial K$. By assumption, $\partial K$ is spectrally closed, and for any point $x \in \partial K$, the closed nc convex hull $M$ generated by $x$ trivially satisfies $M \cap \partial K \ne \emptyset$. Hence by assumption $\partial K \subseteq M$, so Theorem \ref{thm:c-star-simple} implies that $\rmA(K,z)$ is C*-simple.
\end{proof}

%%%
\section{Characterization of C*-algebras} \label{sec:characterization-c-star-algebras}
%%%

A classical result of Bauer characterizes function systems that are unital commutative C*-algebras in terms of their state space. Specifically, he showed that if $C$ is a compact convex set, then the unital function system $\rmA(C)$ of continuous affine functions on $C$ is a unital commutative C*-algebra if and only if $C$ is a Bauer simplex (see e.g.\ \cite{Alf1971}*{Theorem II.4.3}).

The first author and Shamovich \cite{KS2019}*{Theorem 10.5} introduced a definition of noncommutative simplex that generalizes the classical definition and established a generalization of Bauer's result for unital operator systems. Specifically, they showed that if $K$ is a compact nc convex set, then the unital operator system $\rmA(K)$ of continuous affine nc functions on $K$ is a unital C*-algebra if and only if $K$ is an nc Bauer simplex.

In this section we will extend this result by showing that an operator system is a C*-algebra if and only if its nc quasistate space is a Bauer simplex with zero as an extreme point. Before introducing the notion of a Bauer simplex, we need to recall some preliminary definitions.

Let $K$ be a compact nc convex set. For a point $x \in K$, viewed as an nc state on the unital operator system $\rmA(K)$, the *-homomorphism $\delta_x$ is an extension of $x$. We will be interested in other nc states on $\rmC(K)$ that extend $x$. Specifically, we will be interested in nc states that are maximal in a certain precise sense. The following definition is \cite{DK2019}*{Definition 4.5.1}.

\begin{definition}
Let $K$ be a compact nc convex set and let $\mu : \rmC(K) \to \calM_n$ be an nc state on $\rmC(K)$. The {\em barycenter} of $\mu$ is the restriction $\mu|_{\rmA(K)} \in K_n$. The nc state $\mu$ is said to be a {\em representing map} for its barycenter. We will say that a point $x \in K$ has a {\em unique representing map} if the *-homomorphism $\delta_x$ is the unique nc state on $\rmC(K)$ with barycenter $x$. 
\end{definition}

We will also require the notion of a convex nc function. The following definition is \cite{DK2019}*{Definition 3.12}. 

\begin{definition}
Let $K$ be a compact nc convex set. For a bounded self-adjoint nc function $f \in \calM_n(\rmB(K))_h$, the {\em epigraph} of $f$ is the set $\operatorname{Epi}(f) \subseteq \sqcup K_m \times \calM_n(\calM_m)$ defined by
\[
\operatorname{Epi}(f)_m = \{(x,\alpha) \in K_m \times \calM_n(\calM_m) : x \in K_m \text{ and } \alpha \geq f(x)\}.
\]
The function $f$ is {\em convex} if $\operatorname{Epi}(f)$ is an nc convex set.
\end{definition}

Davidson and the first author introduced a notion of nc Choquet order on the set of representing maps of a point in a compact nc convex set that plays a key role in noncommutative Choquet theory. The following definition is \cite{DK2019}*{Definition 8.2.1}.

\begin{definition}
	Let $K$ be a compact nc convex set and let $\mu,\nu : \rmC(K) \to \calM_n$ be nc states. We say that $\mu$ is dominated by $\nu$ in the {\em nc Choquet order} and write $\mu \prec_c \nu$ if $\mu(f) \leq \nu(f)$ for every $n$ and every continuous convex nc function $f \in \calM_n(\rmC(K))$. We will say that $\mu$ is a {\em maximal representing map} for its barycenter if it is maximal in the nc Choquet order.
\end{definition}
\begin{remark}
Several equivalent characterizations of the nc Choquet order were established in \cite{DK2019}. These are among the deepest results in that paper.
\end{remark}

We are finally ready to state the definition of an nc simplex. The following definitions are \cite{KS2019}*{Definition 4.1} and \cite{KS2019}*{Definition 10.1} respectively.

\begin{definition}
\mbox{}
\begin{enumerate}
	\item A compact nc convex set $K$ is an {\em nc simplex} if every point in $K$ has a unique maximal representing map on $\rmC(K)$.
	\item An nc simplex $K$ is an {\em nc Bauer simplex} if the extreme boundary $\partial K$ is a closed subset of the set $\operatorname{Irr}(K)$ of irreducible points in $K$ with respect to the spectral topology.
\end{enumerate}
\end{definition}

\begin{remark}
It was shown in \cite{KS2019} that these definitions generalize the classical definitions. Specifically, if $C$ is a classical simplex then there is a unique nc simplex $K$ with $K_1 = C$. Furthermore, if $C$ is a Bauer simplex then $K$ is an nc Bauer simplex. 
\end{remark}

It was shown in \cite{KS2019}*{Theorem 10.5} that if $K$ is a compact nc convex set, then the unital operator system $\rmA(K)$ is a C*-algebra if and only if $K$ is an nc Bauer simplex. The next example shows that the obvious generalization of this statement for operator systems does not hold.

\begin{example}
Let $(K,z)$ be the pointed compact nc convex set from Example \ref{ex:example-pointed-1}, so $K = \sqcup K_n$ is defined by
\[
K_n = \{\alpha \in (\calM_n)_h : -1_n \leq \alpha \leq 1_n \}, \quad \text{for} \quad n \in \N,
\]
and $z = 0$. 
Since $K_1 = [-1,1]$ is a Bauer simplex, it follows from the above discussion that $K$ is the unique compact nc convex set with this property and $K$ is an nc Bauer simplex. However, $\rmA(K,z)$ is not a C*-algebra. Note that $z \notin \partial K_1$ and hence $z \notin \partial K$.
\end{example}

\begin{lemma} \label{lem:c-star-algebra-iff-unitization-is}
Let $S$ be an operator system with nc quasistate space $(K,z)$. Then $S$ is a C*-algebra if and only if its unitization $S^\sharp$ is a C*-algebra and $z \in \partial K$. 
\end{lemma}

\begin{proof}
If $S$ is a C*-algebra, say $A$, then its unitization $A^\sharp$ is the C*-algebraic unitization $A^\sharp$ of $A$, and hence is also a C*-algebra. Furthermore, $A$ is an ideal in $A^\sharp$ and $z$ is an irreducible *-representation of $A^\sharp$ satisfying $\ker z = A$. Hence by \cite{DK2019}*{Example 6.1.8}, $z \in \partial K$. 

Conversely, suppose that $S^\sharp$ is a C*-algebra, say $B$, and $z \in \partial K$. Then by \cite{DK2019}*{Example 6.1.8}, $z$ is an irreducible representation of $B$. Since $\rmA(K,z) = \ker z$, $\rmA(K,z)$ is an ideal in $B$, and in particular is a C*-algebra.
\end{proof}

The next result extends \cite{KS2019}*{Theorem 10.5}.

\begin{theorem} \label{thm:c-star-alg-iff-nc-bauer-simplex}
Let $S$ be an operator system with nc quasistate space $(K,z)$. Then $S$ is a C*-algebra if and only if $K$ is an nc Bauer simplex and $z \in \partial K$. The result also holds for unital operator systems with nc quasistate spaces replaced by nc state spaces.
\end{theorem}

\begin{proof}
By Lemma \ref{lem:c-star-algebra-iff-unitization-is}, $S$ is isomorphic to a C*-algebra if and only if $S^\sharp$ is isomorphic to a C*-algebra and $z \in \partial K$. By \cite{KS2019}*{Theorem 10.5}, the former property is equivalent to $K$ being a Bauer simplex.
\end{proof}

%%%
\section{Stable equivalence} \label{sec:stable-equivalence}
%%%

Connes and van Suijlekom \cite{CS2020}*{Section 2.6} considered stable equivalence for operator systems. Operator systems $S$ and $T$ are said to be {\em stably equivalent} if the operator systems $S \mintens \calK$ and $T \mintens \calK$ are isomorphic. Here, $\calK = \calK(H_{\aleph_0})$ denotes the C*-algebra of compact operators on $H_{\aleph_0}$ and the minimal tensor products $S \mintens \calK$ and $T \mintens \calK$ are defined as in \cite{KPT2011}, i.e.\ $S \mintens \calK$ is the closed operator system generated by the algebraic tensor product of $S$ and $\calK$ in $\camin(S) \mintens \camin(\calK) = \camin(S) \mintens \calK$ and similarly for $T \mintens \calK$.

In this section we will describe the nc quasistate space of the stabilization of an operator system. This will yield a characterization of stable equivalence in terms of nc quasistate spaces.

Let $S$ be an operator system and let $(K,z)$ denote the nc quasistate space of $S$. Let $(L,w)$ denote the nc quasistate space of $\calK$. For $x \in K$ and $u \in L$, we obtain a completely contractive map $x \otimes u$ on $S \mintens \calK$ from the theory of tensor products of operator spaces (see e.g. \cite{Pis2003}). However, it is not immediately obvious that $x \otimes u$ is an nc quasistate. The next result implies that it is, and moreover, that every nc quasistate on $S \mintens \calK$ arises in this way.

\begin{theorem} \label{thm:tensor-product-nc-quasistates}
	Let $S$ be an operator system with nc quasistate space $(K,z)$ and let $(L,w)$ denote the nc quasistate space of $\calK$. The nc quasistate space of $S \mintens T$ is $(K \otimes L, z \otimes w)$, where $K \otimes L$ denotes the closed nc convex hull of $\{x \otimes u : x \in K \text{ and } u \in L \}$ and $x \otimes u$ is defined as in the above discussion. Furthermore, letting $(M,t)$ denote the nc quasistate space of $S \mintens \calK$, $\partial M \subseteq \partial K \otimes \partial L$.
\end{theorem}

\begin{proof} 
We can identify $S$ with $\rmA(K,z)$ and identify $\rmA(K,z)$ with its image under the canonical embedding into its minimal C*-cover $\rmC(\sb,z)$. By \cite{CS2020}*{Proposition 2.37}, the minimal C*-cover of $\rmA(K,z) \otimes \calK$ is $\rmC(\overline{\partial M}, t) = \rmC(\sb,z) \otimes \calK$. Every point $x \in K_m$ extends to an nc quasistate $\tilde{x} : \rmC(\sb,z) \to \calM_m$ (see Section \ref{sec:min-max-c-star-covers}). Then for $u \in L_n$, we obtain an nc quasistate $\tilde{x} \otimes y : \rmC(\sb,z) \otimes \calK \to \calM_m \otimes \calM_n$. The restriction $\tilde{x} \otimes y|_{\rmA(K,z) \otimes \calK} = x \otimes y$ is therefore an nc quasistate on $\rmA(K,z) \mintens \calK$. Hence $K \otimes L \subseteq M$. It is clear that $z \otimes w$ is the zero map.

For the reverse inclusion, let $r \in \partial M$ be an extreme point. Then by Proposition \ref{prop:factors-through}, the *-homomorphism $\delta_r$ on $\rmC(M,t)$ factors through $\rmC(\sb,z) \otimes \calK$. Since $r$ is extreme, \cite{DK2019}*{Theorem 6.1.9} implies that $\delta_r$ is irreducible. Hence there is an irreducible representation $\pi : \rmC(\sb,z) \to \calM_m$ such that $\delta_r$ is unitarily equivalent to $\pi \otimes u$, where $u \in L$ is either the identity representation of $\calK$ or $u = w$. Letting $x = \pi|_{\rmA(K,z)} \in K_m$, $r|_{\rmA(K,z) \mintens \calK} = x \otimes u$. In fact, it is easy to verify that since $r \in \partial L$, $x \in \partial K$. It follows from the nc Krein-Milman theorem \cite{DK2019}*{Theorem 6.4.2} that $M \subseteq K \otimes L$.
\end{proof}

\begin{corollary}
Let $S_1$ and $S_2$ be operator systems with nc quasistate spaces $(K_1,z_1)$ and $(K_2,z_2)$ respectively. Let $0_\calK$ and $\id_{\calK}$ denote the zero map and the identity representation respectively of $\calK$. Then $S$ and $T$ are stably isomorphic if and only if the closed nc convex hulls of the sets $\partial K \otimes \{0_{\calK}, \id_{\calK}\}$ and $\partial L \otimes \{0_{\calK}, \id_{\calK}\}$ are pointedly affinely homeomorphic with respect to the points $z_1 \otimes 0_\calK$ and $z_2 \otimes 0_\calK$.
\end{corollary}

\begin{proof}
Let $(L,0_{\calK})$ denote the nc quasistate space of $\calK$. Then it follows from Theorem \ref{thm:tensor-product-nc-quasistates} and Corollary \ref{cor:isomorphic-iff-affinely-homeomorphic} that $S$ and $T$ are stably isomorphic if and only if $(K_1 \otimes L,z_1 \otimes 0_{\calK})$ and $(K_2 \otimes L, z_2 \otimes 0_{\calK})$ are pointedly affinely homeomorphic. 

By Proposition \ref{prop:nc-state-space-unitization}, the unitization $\calK^\sharp$ is a unital C*-algebra with nc state space $L$. Since every irreducible *-representation of $\calK^\sharp$ is unitarily equivalent to $0_{\calK}$ or $\id_{\calK}$, \cite{DK2019}*{Example 6.1.8} implies that $L$ is the closed nc convex hull of $\{0_{\calK}, \id_{\calK}\}$. The result now follows from Theorem \ref{thm:tensor-product-nc-quasistates} and the nc Krein-Milman theorem \cite{DK2019}*{Theorem 6.4.2}.
\end{proof}

%%%
\section{Dynamics and Kazhdan's property (T)} \label{sec:dynamics}
%%%

The fact that simplices arise as fixed point sets of affine actions of groups on spaces of probability measures has a number of important applications in classical dynamics. Glasner and Weiss showed that a second countable locally compact group has Kazhdan's property (T) if and only if the simplices that arise from this result are always Bauer simplices \cite{GW1997}.

The first author and Shamovich extended these results to actions of discrete groups on nc state spaces of unital C*-algebras. Specifically, it was shown that nc simplices arise as fixed point sets of affine actions of discrete groups on nc state spaces of unital C*-algebras \cite{KS2019}*{Theorem 12.12}. It was further shown that a discrete group has property (T) if and only if the nc simplices that arise from this result are always nc Bauer simplices \cite{KS2019}*{Theorem 14.2}. Consequently, a discrete group has property (T) if and only if whenever it acts on a unital C*-algebra, the set of invariant states is the state space of a unital C*-algebra \cite{KS2019}*{Corollary 14.3}. 

In this section we will extend these results to actions of locally compact groups on (potentially nonunital) C*-algebras. In fact, we will see that the hard work was already accomplished in earlier sections of this paper. After introducing appropriate definitions and applying the dual equivalence between the category of operator systems and the category of pointed compact nc convex sets, the proofs in \cite{KS2019} will apply essentially verbatim. 

The next definition is a slight generalization of \cite{KS2019}*{Definition 12.1} and \cite{KS2019}*{Definition 12.2}.

\begin{definition}
\mbox{}
\begin{enumerate}
	\item An {\em nc dynamical system} is a triple $(S,G,\sigma)$ consisting of an operator system $S$, a locally compact group $G$ and a group homomorphism $\sigma : G \to \Aut(S)$ with the property that the orbit map $G \to S : g \to \sigma_g(s)$ is continuous for all $s \in S$.
	\item A {\em affine nc dynamical system} is a triple $(K,G,\kappa)$ consisting of a compact nc convex set $K$, a locally compact group $G$ and a group homomorphism $\kappa : G \to \Aut(K)$ with the property that for each $n$, the orbit map $G \to K_n : g \to \kappa_g(x)$ is continuous for all $x \in K_n$.
\end{enumerate}
\end{definition}

\begin{remark}
Unless we need to refer to $\sigma$, we will write $(S,G)$ for $(S,G,\sigma)$ and $g s$ for $\sigma_g(s)$. Similarly, unless we need to refer to $\kappa$, we will write $(K,G)$ for $(K,G,\kappa)$ and $g x$ for $\kappa_g(x)$. If $S$ is a C*-algebra, say $A$, then we will refer to $(A,G)$ as a {\em C*-dynamical system}.
\end{remark}

We will utilize the fact that if $(K,z)$ is a pointed compact nc convex set and $(\rmA(K,z),G)$ is an nc dynamical system, then the dual equivalence from Theorem \ref{thm:categorical-equivalence} gives rise to an affine nc dynamical system $(K,G)$, determined by
\[
a(\kappa_g(x)) = \sigma_{g^{-1}}(a)(x), \quad \text{for} \quad a \in \rmA(K),\ g \in G \text{ and } x \in K.
\]

It seems worth pointing out that an nc dynamical system over an operator system lifts to an nc dynamical system on its unitization.

\begin{lemma}
	Let $(S,G,\sigma)$ be an nc dynamical system. Define $\sigma^\sharp : G \to \Aut(S)$ by $(\sigma^\sharp)_g = (\sigma_g)^\sharp$. Then $(S^\sharp,G,\sigma^\sharp)$ is an nc dynamical system.
\end{lemma}

\begin{proof}
For $g \in G$, $(\sigma_g)^\sharp (s,\alpha) = (\sigma_g(s),\alpha)$ for $s \in S^\sharp$. It follows immediately that  $\sigma^\sharp : G \to \Aut(S^\sharp)$ is a group homomorphism and that the corresponding orbit maps are continuous.
\end{proof}

Let $G$ be a locally compact group. Recall that a {\em continuous} unitary representation of $G$ on a Hilbert space $H$ is a group homomorphism $\rho : G \to \mathcal{U}(H)$ such that the orbit map $G \to H: g \to \rho(g) \xi$ is continuous for every $\xi \in H$. Here $\mathcal{U}(H)$ denotes the set of unitary operators on $H$. 

The next result follows immediately from \cite{KS2019}*{Theorem 12.12}, since we can view the action of a non-discrete locally compact group as an action by its discretization.

\begin{theorem} \label{thm:fixed-point-set-nc-simplex}
Let $(K,G)$ be an affine nc dynamical system such that $K$ is an nc simplex. Then the fixed point set
\[
K^G = \{x \in K : gx = x \text{ for all } g \in G \}
\]
is an nc simplex. 
\end{theorem}

\begin{corollary} \label{cor:fixed-point-set-nc-simplex}
Let $(A,G)$ be a C*-algebra and let $(K,z)$ denote the nc quasistate space of $A$. Then the fixed point set $K^G$ is an nc simplex.
\end{corollary}

\begin{proof}
By Theorem \ref{thm:c-star-alg-iff-nc-bauer-simplex}, $K$ is an nc Bauer simplex, so the result follows immediately from Theorem \ref{thm:fixed-point-set-nc-simplex}.	
\end{proof}

\begin{definition}
Let $G$ be a second countable locally compact group.
\begin{enumerate}
\item A continuous unitary representation $\rho : G \to \mathcal{U}(H)$ is said to have {\em almost invariant vectors} if there is a net of unit vectors $\{\xi_i \in H\}$ such that for every compact subset $C \subseteq G$,
\[
\lim_i \sup_{g \in C} \|\rho(g)\xi_i - \xi_i\| = 0.
\]
\item The group $G$ is said to have {\em Kazhdan's property (T)} if every unitary representation of $G$ with almost invariant vectors has a nonzero invariant vector. 
\end{enumerate}
\end{definition}

The next result is a generalization for (potentially nonunital) C*-algebras and second countable locally compact groups of \cite{KS2019}*{Theorem 14.2}.

\begin{theorem} \label{thm:property-t}
Let $A$ be a C*-algebra with nc quasistate space $(K,z)$ and let $G$ be a second countable locally compact group with Kazhdan's property (T) such that $(A,G)$ is a C*-dynamical system. The set $K^G$ of invariant nc quasistates on $A$ is an nc Bauer simplex. If $A$ is unital, then the result also holds for the nc state space of $A$ instead of its nc quasistate space.
\end{theorem}

\begin{proof}
The proof of \cite{KS2019}*{Theorem 14.2} works essentially verbatim here. If $G$ is non-discrete, then it is necessary to verify that the unitary representation constructed in the proof of the dilation theorem for invariant nc states \cite{KS2019}*{Lemma 12.6} is continuous. However, this is an easy consequence of the continuity of the orbit maps.
\end{proof}

The following corollary extends a result of Glasner and Weiss for commutative C*-algebras (see \cite{GW1997}*{Theorem 1'} and \cite{GW1997}*{Theorem 2'}).

\begin{corollary}
Let $G$ be a second countable locally compact group. Then $G$ has Kazhdan's property (T) if and only if whenever $A$ is a C*-algebra with nc quasistate space $(K,z)$ and $(A,G)$ is a C*-dynamical system, then the set $K_1^G$ of invariant quasistates is pointedly affinely homeomorphic to the quasistate space of a C*-algebra. If $A$ is unital, then the result also holds with the quasistate space of $A$ replaced by its state space.
\end{corollary}

\begin{proof}
If $G$ has Kazhdan's property (T), then Theorem \ref{thm:property-t} implies that $K^G$ is an nc Bauer simplex. By Lemma \ref{lem:c-star-algebra-iff-unitization-is}, $z \in \partial K$. Hence by Theorem \ref{thm:c-star-alg-iff-nc-bauer-simplex}, $(K^G,z)$ is pointedly affinely homeomorphic to the nc quasistate space of a C*-algebra. In particular, the set $K_1^G$ of invariant quasistates of $A$ is pointedly affinely homeomorphic to the quasistate space of a C*-algebra.

Conversely, if $G$ does not have Kazhdan's property (T), then it follows from \cite{GW1997}*{Theorem 2'} that there is a compact Hausdorff space $X$ and a commutative C*-dynamical system $(\rmC(X),G)$ such that the space $\operatorname{Prob}(X)^G$ of invariant probability measures on $X$ is a Poulsen simplex. Equivalently, the set $\partial (\operatorname{Prob}(X)^G)$ of extreme points of $\operatorname{Prob}(X)^G$ is not closed. 

We need to translate this to a statement about the quasistate space $Q$ of $\rmC(X)$. Since $Q$ is a compact convex set, the set $Q^G$ of invariant quasistates is a simplex (see e.g.\ \cite{KS2019}*{Corollary 12.13}). Note that $\operatorname{Prob}(X) \subseteq Q$. In fact, $Q$ is the closed convex hull of $\operatorname{Prob}(X) \cup \{z\}$, where $z$ denotes the zero map on $\rmC(X)$. For nonzero $\mu \in \partial (Q^G)$, since $\mu(X)^{-1} \mu \in Q^G$, it follows that $\mu(X) = 1$. Hence $\mu \in \partial (\operatorname{Prob}(X)^G)$. On the other hand, it is clear that $\partial (\operatorname{Prob}(X)^G) \subseteq \partial (Q^G)$. Hence $\partial (Q^G) \subseteq \partial (\operatorname{Prob}(X)^G) \cup \{z\}$. Since $\partial (\operatorname{Prob}(X)^G)$ is not closed and $z$ is isolated from $\operatorname{Prob}(X)$, it follows that $\partial (Q^G)$ is not closed. Therefore, $Q^G$ is not a Bauer simplex.

The result now follows from the fact that if the quasistate space of a C*-algebra (equivalently, the state space of its unitization) is a simplex, then the C*-algebra is commutative and its quasistate space is a Bauer simplex (see e.g.\ Theorem \ref{thm:c-star-alg-iff-nc-bauer-simplex}). 
\end{proof}

%%%
% Bibliography
%%%

\end{document}